\documentclass{birkjour}

\usepackage[utf8]{inputenc}
\usepackage[T1]{fontenc}
\usepackage[english]{babel}
\usepackage{mathtools,amsmath,amssymb,amsfonts,amsthm,mathrsfs}
\mathtoolsset{showonlyrefs}
\usepackage{enumitem}

\usepackage{geometry}
\usepackage{graphicx,color}
\usepackage{tikz,pgfplots,pgf}
\usepackage[font=small]{caption}

\newtheorem{theorem}{Theorem}[section]
\newtheorem{lemma}{Lemma}[section]

\newtheorem{corollary}{Corollary}[section]
\theoremstyle{definition}
\newtheorem{example}{Example}[section]
\theoremstyle{definition}
\newtheorem{remark}{Remark}[section]
\newtheorem{definition}{Definition}[section]

\numberwithin{equation}{section}

\begin{document}

\title[Equilibrium points, periodic solutions and the Brouwer fixed point theorem]{Equilibrium points, periodic solutions \\and the Brouwer fixed point theorem \\for convex and non-convex domains}

\author[G.~Feltrin]{Guglielmo Feltrin}

\address{
Department of Mathematics, Computer Science and Physics, University of Udine\\
Via delle Scienze 206, 33100 Udine, Italy}

\email{guglielmo.feltrin@uniud.it}

\author[F.~Zanolin]{Fabio Zanolin}

\address{
Department of Mathematics, Computer Science and Physics, University of Udine\\
Via delle Scienze 206, 33100 Udine, Italy}

\email{fabio.zanolin@uniud.it}

\thanks{Work performed under the auspices of the Grup\-po Na\-zio\-na\-le per l'Anali\-si Ma\-te\-ma\-ti\-ca, la Pro\-ba\-bi\-li\-t\`{a} e le lo\-ro Appli\-ca\-zio\-ni (GNAMPA) of the Isti\-tu\-to Na\-zio\-na\-le di Al\-ta Ma\-te\-ma\-ti\-ca (INdAM).
\\
\textbf{Preprint -- March 2022}}

\subjclass{34C25, 34A26, 52A20.}

\keywords{Equilibrium points, periodic solutions, positively invariant sets, Brouwer fixed point theorem, convex sets, star-shaped sets.}

\date{}

\dedicatory{}

\begin{abstract}
We show the direct applicability of the Brouwer fixed point theorem for
the existence of equilibrium points and periodic solutions for differential systems
on general domains satisfying geometric conditions at the boundary.
We develop a general approach for arbitrary bound sets and present applications to the case of convex and star-shaped domains.
We also provide an answer to a question raised in a recent paper of Cid and Mawhin.
\end{abstract}

\maketitle

\dedicatory{Dedicated to Professor Jean Mawhin for his coming 80th birthday}

\section{Introduction}\label{section-1}

The Brouwer fixed point theorem can be well considered as a
fundamental result in the area of topological fixed point
theorems. As clearly explained in the recent book of Dinca and
Mawhin \cite{DiMa-2021} and developed in the content of the book,
as well as in the survey articles \cite{Ma-2020, Pa-1999} and the
references therein, Brouwer theorem can be seen as a ``core
result'' from which several important theorems about the existence
of fixed points and periodic points can be proved. Further
connections between the Brouwer fixed point theorem and other
results (including theorems of combinatorial nature) can be found
in \cite{Ga-1979, PaJe-2003} and \cite{IKM-2014, IKM-2021, Pa-1989}.
From this perspective, an interesting line of research, already
pursued by several authors, is to establish different results of
Nonlinear Analysis and its application to differential equations
which can be obtained, more or less directly, from this theorem.

In the present paper, we follow such a point of view, in the
context of the theory of positively invariant sets for differential systems
in $\mathbb{R}^{N}$.
A classical general result from Dynamical Systems asserts that
if a positively invariant set is homeomorphic to a closed ball of
$\mathbb{R}^{N}$, then it contains an equilibrium point (see
Theorem~\ref{th-2.2}). An analogous result can be proved for
\textit{processes} (see \cite[Chapter~4]{Ha-1977} for the definition
and general properties of processes), guaranteeing the existence
of periodic points for periodic processes in positively invariant
sets with the fixed point property. All these results, when
applied to ordinary differential systems, require the uniqueness
(or at least forward uniqueness) of the solutions for the
associated initial value problems. When uniqueness is not assumed,
the study becomes technically more involved, because the associated
Poincar\'{e} map is multivalued \cite{AnGo-2006,DyGo-1983} and the classical general
Nagumo--Yorke--Bony--Brezis-type theorems, which are expressed by
sub-tangential conditions at the boundary \cite{Bo-1969,Br-1970,Na-1942,Yo-1967}, guarantee a weak form of
positively invariance, namely, that for each point in a given set
$M$ \textit{at least one solution} remains in $M$ for positive
time. Hence, if one considers a (multivalued) Poincar\'{e} map
$\Phi$, it occurs that $\Phi(u)\cap M\neq\emptyset$, for all
$u\in M$, instead of $\Phi(u)\subset M$, for all $u\in M$.

When $M\subset \mathbb{R}^{N}$ is a convex set and $f \colon \mathbb{R}^{N}\to \mathbb{R}^{N}$
is a continuous vector field, the sub-tangential condition for the weak flow-invariance for the differential system $\dot y = f(y)$ is expressed by the relation
\begin{equation}\label{eq-1.1}
\langle f(u),\nu\rangle \leq 0,\quad \text{for all $u\in \partial M$ and $\nu \in N(u)$,}
\end{equation}
where $N(u)$ denotes the set of all outer normals to $M$ at the point $u$.
Recently, Cid and Mawhin in \cite{CiMa-2021} provided a new proof of the existence of equilibrium points
for $f$ in a convex set $M$, under assumption \eqref{eq-1.1}, as a direct application
of the Brouwer theorem. Actually, the results in \cite{CiMa-2021} apply also to the existence
of $T$-periodic solutions in $M$ for non-autonomous vector fields. Related results can be also found in
\cite{FoGi-2016} using degree theory.
In \cite[Section 5]{CiMa-2021} it was raised
the question whether the Brouwer fixed point theorem can be still
applied if \eqref{eq-1.1} is relaxed to a similar condition where
the inequality holds only \textit{for some} (and not all) $\nu\in
N(u)$, when $M$ is a convex sets with nonempty interior.

A first contribution of the present paper is to provide a positive answer to the above question
(see Theorem~\ref{th-cm1}). This is illustrated in Section~\ref{section-2}, where we also
recall some basic concepts as well as a few classical facts from dynamical systems theory
which are then used in the proofs.
In our study, we adopt an approach previously introduced in
\cite{Ma-1977}, in the framework of the theory of \textit{bound sets},
and further developed in \cite{FeZa-1988} in order to obtain results of
\textit{strong flow-invariance}
(namely for each point in a given set, \textit{all the solutions} departing from the
point remain in the sets in future time). To show the effectiveness of this approach,
in Section~\ref{section-3}, following \cite{FeZa-1988},
we present our main result (cf. Theorem~\ref{th-main}), which ensures the existence of periodic
solutions for non-autonomous systems in general bound sets,
by assuming the existence of an auxiliary vector field
satisfying a suitable ``condition of inwardness''  \eqref{cond-th-main} at the boundary.
Then, with the aid of such a condition,
we further investigate applications to non-convex domains. In particular, we will provide
new results for star-shaped sets, using geometric hypotheses on Bony outer normals.

\smallskip

Throughout the article, the following notation is used. We denote with $\|\cdot\|$ the Euclidean norm in $\mathbb{R}^{N}$. The symbols $B(x_{0}, r)$ and $B[x_{0},r]$, where $x_{0}\in\mathbb{R}^{N}$ and $r>0$, represent the open
and closed balls centered at $x_{0}$ with radius $r$ respectively, i.e.~$B(x_{0},r) = \{x\in\mathbb{R}^{N} \colon \|x-x_{0}\|<r\}$ and $B[x_{0},r] = \{x\in\mathbb{R}^{N} \colon \|x-x_{0}\|\leq r\}$. Moreover, $\mathbb{S}^{N-1} :=\partial B(0,1) \subset \mathbb{R}^{N}$.
Given a subset $D\subset \mathbb{R}^{N}$,
we denote with $\mathrm{int\,}D$ its interior, with $\overline{D}$ its closure, and with $\partial D$ its boundary.

\section{Equilibrium points in convex bodies: a dynamical systems approach}\label{section-2}

In this section we study, from a dynamical point of view, the problems of existence of equilibrium points (zeros of vector fields) in compact convex sets with nonempty interior. Our aim is to show how
existence results can be obtained as a direct application of the Brouwer fixed point theorem, in the framework of the theory of positively invariant sets.
The search of zeros of a vector field $f$, or, equivalently,
equilibria for the differential system
\begin{equation}\label{eq-2.1}
\dot{y}= f(y),
\end{equation}
is a classical and well investigated topic in the area of
Nonlinear Analysis, including important applications to different
disciplines, like control theory \cite{Bl-1999,BlMi-2015,Fr-2018}.

\subsection{Preliminaries on dynamical systems}\label{section-2.1}

Let $X$ be a metric space. A \textit{dynamical system} on $X$ is a \textit{continuous} map $\Pi \colon X\times \mathbb{R}\to \mathbb{R}$ such that
\begin{itemize}
\item[$(a_{0})$] $\Pi(x,0) = x$, for every $x\in X$;
\item[$(a_{1})$] $\Pi(\Pi(x,s),t) = \Pi(x,s+t)$, for every $x\in X$ and $s,t \in \mathbb{R}$.
\end{itemize}
Following \cite{BhSz-1970}, we will use the simplified notation
\begin{equation*}
x\cdot t = \Pi(x,t), \quad \Pi^{t} \colon x \mapsto x\cdot t.
\end{equation*}
Moreover, when no confusion may occur, the dot ``$\, \cdot \,$'' will be omitted. In this manner, $(a_{0})$ and $(a_{1})$ can be written equivalently, as
\begin{itemize}
\item[$(a_{0})$] $x0 = x$ (respectively $\Pi^0 = \mathrm{Id}_X$), for every $x\in X$;
\item[$(a_{1})$] $(xs)t= x(s+t)$ (respectively $\Pi^{t}\circ \Pi^{s} = \Pi^{s+t}$), for every $x\in X$ and $s,t \in \mathbb{R}$;
\end{itemize}
while the continuity of the dynamical system can be expressed in sequential form as
\begin{equation*}
x_{n}\to_{_X} x , \quad  t_{n}\to_{_\mathbb{R}} t \quad \Longrightarrow \quad x_{n}t_{n} \to_{_X} xt.
\end{equation*}
An \textit{equilibrium point} or \textit{rest point} (or \textit{critical point}, according to \cite{BhSz-1970}) for the dynamical system is a point $z\in X$ such that $zt=z$ for every $t\in \mathbb{R}$.
By a \textit{periodic point} we mean a point $w\in X$ such that there exists $\tau\neq0$ such that $w\tau = w$. Without loss of generality, we can suppose that $\tau>0$.
It is a standard fact that if $w$ is a periodic point and not an equilibrium, then there exists a minimal period (\textit{fundamental period}) $T>0$ such that any period of $w$ is an integer multiple of $T$ (cf.~\cite[Theorem~2.12]{BhSz-1970}).
Moreover, any point $x$, belonging to the orbit $\gamma(w):= \{ wt \colon t\in \mathbb{R}\}$ of $w$, is also a periodic point with the same period and $\gamma(w)$ is homeomorphic to $\mathbb{S}^{1}=\{(x,y)\in \mathbb{R}^{2} \colon x^{2}+y^{2}=1\}= \partial B(0,1)$.

If the dynamical system is the flow associated with the autonomous differential equation
\eqref{eq-2.1}
where $f \colon \Omega \to \mathbb{R}^{N}$ is a continuous vector field defined on a nonempty open set $\Omega \subset \mathbb{R}^{N}$ and such that the (local) uniqueness for the initial value problems holds, then $z\in \Omega$ is an equilibrium point if and only if $f(z)=0$. Technically, if the global existence of the solutions for \eqref{eq-2.1} is not guaranteed, we should speak of a \textit{local dynamical system} associated with \eqref{eq-2.1} (see \cite{LaS-1976}), however, Vinograd's Theorem \cite{BhSz-1970, NeSt-1960} allows to enter in the setting of dynamical systems even in absence of global existence of the solutions for \eqref{eq-2.1}. In any case, all the results we are going to present in this article, even if restricted to the case of dynamical systems, apply to local dynamical systems as well (see also \cite{Ca-1972, Ur-1969} for a discussion about the
reparametrization problem). The crucial step to define a (local) dynamical system associated with \eqref{eq-2.1} comes from the assumption of the uniqueness for the solutions of the Cauchy problems and, consequently, from the theorem of continuous dependence of the solutions from initial data (cf.~\cite{Ha-1980, NeSt-1960, Se-1971}). From this, we can set
\begin{equation*}
\Pi^{t}(x) := \tilde{y}(t,x),
\end{equation*}
where $\tilde{y}= \tilde{y}(\cdot,x)$ is the solutions of \eqref{eq-2.1} with $y(0)=x\in \Omega$.

The following lemma is borrowed from \cite[Lemma~2.15, p.~18]{BhSz-1970}.

\begin{lemma}\label{lem-2.1}
For a dynamical system on a metric space $X$, let $(x_{n})_{n}$ be a sequence of $T_{n}>0$ periodic points such that $x_{n}\to z$ in $X$ and $T_{n}\to 0^{+}$. Then, $z$ is an equilibrium point.
\end{lemma}

For the next results, we need to introduce two further definitions.
Let $M\subset X$ be a subset of a metric space with a dynamical system $\Pi$. We say that $M$ is \textit{positively invariant} (or \textit{flow-invariant}) if for each $x\in M$ we have $xt\in M$ for all $t\geq 0$. Similarly, we define a \textit{negatively invariant} set, by the same relation with $t\leq 0$. Finally, $M$ is an \textit{invariant set} if it is both positively and negatively invariant, namely $xt\in M$ for all $t\in \mathbb{R}$ whenever $x\in M$. According to \cite[Chapter~2]{BhSz-1970}, we have the following result.

\begin{lemma}\label{lem-2.2}
Let $X$ be a metric space with a dynamical system and let $M\subset X$. If $M$ is positively/negatively invariant, then
$X\setminus M$ is negatively/positively invariant. Moreover, also $\overline{M}$ and $\mathrm{int\,}M$ are positively/negatively invariant.
If $M$ is open or closed and $\partial M$ is invariant, also $M$ is invariant.
\end{lemma}

A topological space $Y$ has the \textit{Fixed Point Property} (FPP) if any continuous map $\psi \colon Y\to Y$ has (at least) a fixed point. The (FPP) is invariant under homeomorphisms and is preserved under retractions on subspaces (see \cite{Bi-1969} for more information). Hence, as a consequence of the Brouwer fixed point theorem, we can state the following (cf.~\cite[Theorem~3]{Ma-2020}).

\begin{lemma}\label{rem-2.1}
Let $X\subset \mathbb{R}^{N}$ be a compact set which is homeomorphic to a (retract of a) closed ball of $\mathbb{R}^{N}$. Then, $X$ has the \emph{(FPP)}.
\end{lemma}

In particular, any (nonempty) closed bounded set in $\mathbb{R}^{N}$ has the (FPP) (which is an equivalent formulation of the Brouwer fixed point theorem for compact convex sets); if $\Gamma\subset \mathbb{R}^{2}$ is a Jordan curve (the homeomorphic image of $\mathbb{S}^{1}$) with $A_{\mathrm{int}}(\Gamma)$ and $A_{\mathrm{ext}}(\Gamma)$ its interior and exterior open domains, then $\mathcal{D}(\Gamma):=A_{\mathrm{int}}(\Gamma)\cup \Gamma = A_{\mathrm{int}}(\Gamma)
\cup \partial A_{\mathrm{int}}(\Gamma)$ is homeomorphic to $B[0,1]$ (by the Jordan--Sch\"{o}nflies theorem \cite{Mo-1977}) and hence it has the (FPP).

An application of Lemma~\ref{lem-2.1} gives the following.

\begin{theorem}\label{th-2.2}
Let $M$ be a (nonempty) compact positively/negatively invariant in a metric space with a dynamical system. If $M$ has the \emph{(FPP)}, then it contains an equilibrium point.
\end{theorem}

\begin{proof}
Let $M$ be positively invariant (for the negatively invariant sets a similar argument works). We take a decreasing sequence of positive real numbers $\tau_{n}$ with $\tau_{n}\searrow 0$ and, for each $n$, consider the continuous map $\Pi_{n}:=\Pi^{\tau_{n}}$. By the positive invariance of $M$ we have that $\Pi_{n}\colon M\to M$ and, by the (FPP), there exists at least one fixed point $y_{n}\in M$ for $\Pi_{n}$, so that $y_{n}=y_{n} \tau_{n}$. By the compactness of $M$ we can find a subsequence $x_{n}:=y_{k_{n}}$ with $x_{n}\to x^{*}\in M$. Setting $T_{n}:= \tau_{k_{n}}$, we have that $(x_{n})$ is a sequence
of periodic points in $M$ with periods $T_{n}\to 0^{+}$ and Lemma~\ref{lem-2.1} ensures that $x^{*}$ is an equilibrium point.
\end{proof}

As an application of Theorem~\ref{th-2.2} we can obtain a classical result for planar dynamical systems, according to which, if $\Gamma$ is a closed orbit of a dynamical system in the plane, then $A_{\mathrm{int}}(\Gamma)$ contains an equilibrium point.
Indeed, by the last statement in Lemma~\ref{lem-2.2}, $\mathcal{D}(\Gamma)$ is an invariant set, since
$\Gamma=\partial A_{\mathrm{int}}(\Gamma)$ is invariant (in fact it is an orbit). Hence, there is an equilibrium point $x^{*}$ in $\mathcal{D}(\Gamma)$, with $x^{*}\notin \Gamma$ and the claim is proved.

\subsection{Zeros of vector fields and convex bodies}\label{section-2.2}

From now on, we focus our attention on the (local) dynamical systems defined by ordinary differential equations in $\mathbb{R}^{N}$. An immediate consequence of Theorem~\ref{th-2.2} is the following.

\begin{corollary}\label{cor-2.1}
Let $\Omega\subset \mathbb{R}^{N}$ be an open set and $f \colon \Omega\to \mathbb{R}^{N}$ be a continuous map. Let $C\subset \Omega$ be a compact set with the \emph{(FPP)}. Assume that:
\begin{itemize}[leftmargin=38pt,labelsep=6pt]
\item[$(H_{\textrm{uniq}})$] all the Cauchy problems associated with \eqref{eq-2.1} have a unique solution;
\item[$(H_{\textrm{inv}})$]  the set $C$ is positively invariant with respect to the solutions of \eqref{eq-2.1}, namely, for each $x_{0}\in C$, the solution $\tilde{y}$ of \eqref{eq-2.1} with $\tilde{y}(0)= x_{0}$ is such that $\tilde{y}(t)\in C$ for all $t\geq 0$ in the right maximal interval of existence of the solution.
\end{itemize}
Then, there exists $z\in C$ with $f(z)=0$.
\end{corollary}

Now, the question that arises is whether Corollary~\ref{cor-2.1}
can be improved by removing the hypothesis of uniqueness for the initial value problems. An interesting step in this direction has been recently achieved by Cid and Mawhin in \cite{CiMa-2021}, where, as a consequence of a more general result on periodic solutions for non-autonomous systems, the following result is proved (cf.~\cite[Corollary~5.1]{CiMa-2021})

\begin{theorem}\label{th-cm}
Let $C \subset \mathbb{R}^{N}$ be a nonempty, closed, bounded and convex set,
and $f \colon C \to  \mathbb{R}^{N}$ be a continuous function.
Suppose that, for each outer normal field
$\nu \colon \partial C \to \mathbb{S}^{N-1}$, it holds that
\begin{equation}\label{eq-cm}
\langle f(x), \nu(x) \rangle \leq 0, \quad \text{for all $x\in \partial C$.}
\end{equation}
Then, the differential equation \eqref{eq-2.1} has a constant solution
in $C$.
\end{theorem}

Assumption \eqref{eq-cm} corresponds to an equivalent formulation, for the case of convex sets, of the classical Nagumo condition \cite{Na-1942} for the weak flow-invariance of the set $C$ (see also \cite{Bo-1969, Br-1970, Cr-1972, Ha-1972, Re-1972, ReWa-1975, Yo-1967}).
In this context, the term \textit{weak flow-invariance} refers to the fact that any Cauchy problem with initial value in $C$ has \textit{at least} one local solution which is in $C$ for positive time. Hence $(H_{\mathrm{inv}})$ is satisfied if we further assume the uniqueness of the solutions for the initial value problems.
Notice, however, that Theorem~\ref{th-cm} is obtained (from the Brouwer fixed point theorem) without assuming the uniqueness hypothesis $(H_{\textrm{uniq}})$.

On the other hand, in \cite[Corollary~3]{FeZa-1988} the weak positive invariance is obtained under a more general geometric boundary condition,
at the expense of a stronger assumption on the convex set (which is required to have a nonempty interior),
as follows.

\begin{theorem}\label{th-fz}
Let $C \subset \mathbb{R}^{N}$ be a nonempty, closed, bounded and convex set with nonempty interior, and $f \colon C \to  \mathbb{R}^{N}$ be a continuous function. Suppose that there exists an outer normal field $\nu \colon \partial C \to \mathbb{S}^{N-1}$, such that condition \eqref{eq-cm} holds. Then, $C$ is a weakly invariant set for \eqref{eq-2.1}.
\end{theorem}

According to this result, we have that $(H_{\textrm{inv}})$ is satisfied, if $(H_{\textrm{uniq}})$ holds.

Our aim is to show that, under the same assumptions of Theorem~\ref{th-fz} (and without invoking the uniqueness of the solutions for the initial value problems), one can also conclude the existence of an equilibrium point in $C$.
This will follow from a direct application of the Brouwer fixed point theorem. Namely, the following result holds.

\begin{theorem}\label{th-cm1}
Let $C \subset \mathbb{R}^{N}$ be a nonempty, closed, bounded and convex set with nonempty interior, and $f \colon C \to  \mathbb{R}^{N}$ be a continuous function. Suppose that there exists an outer normal field $\nu \colon \partial C \to \mathbb{S}^{N-1}$,
such that condition \eqref{eq-cm} holds. Then, the differential equation \eqref{eq-2.1} has a constant solution in $C$.
\end{theorem}

\begin{proof}
We follow an approximation scheme already applied in the proof of \cite[Corollary~3]{FeZa-1988}. Without loss of generality, we can assume that $f \colon \mathbb{R}^{N}\to  \mathbb{R}^{N}$ is continuous. Let $w_{0}\in \mathrm{int\,}C$ and $r>0$
such that $B[w_{0},r]\subset C$. For each $x\in \partial C$ and $\nu(x)\in \mathbb{S}^{N-1}$, we have that $w_{0} + r\nu(x)\in B[w_{0},r]$; therefore $\langle w_{0} + r\nu(x) -x,\nu(x)\rangle \leq 0$ and so $\langle w_{0}  -x,\nu(x)\rangle \leq -r \|\nu(x)\|^{2}\leq -r$.
Next, we consider a sequence of locally Lipschitz vector fields $g_{n} \colon \mathbb{R}^{N}\to  \mathbb{R}^{N}$ converging uniformly to $f$ to compact sets. Hence, passing if necessary to a subsequence, we can suppose that $\|g_{n}(x)-f(x)\|\leq \frac{r}{n}$, for every $x\in C$. Finally, we define the vector fields
\begin{equation*}
f_{n}(x):= g_{n}(x) + \dfrac{1}{n}(w_{0}-x), \quad x\in \mathbb{R}^{N},
\end{equation*}
and observe that, for the given outer normal field $\nu$,
\begin{align*}
\langle f_{n}(x),\nu(x)\rangle &= \langle f(x),\nu(x)\rangle
+ \langle g_{n}(x) - f(x),\nu(x)\rangle + \frac{1}{n}\langle w_{0}-x,\nu(x)\rangle
\\
&\leq 0 + \|g_{n}(x)-f(x)\| - \frac{r}{n} \leq 0, \quad \text{for every $x\in \partial C$.}
\end{align*}
Applying Theorem~\ref{th-fz} to the differential systems $\dot{y}= f_{n}(y)$ for which we have the uniqueness of the solutions for the initial value problems, we enter in the setting of Corollary~\ref{cor-2.1} and, therefore, for each $n$ there exists at least a $z_{n}\in C$ such that $f_{n}(z_{n})=0$. Finally, by the compactness of $C$ we have that a subsequence $z_{k_{n}}$ converges to a point $z^{*}\in C$ and $f(z^{*})=0$, since $(f_{n})_{n}$ converges to $f$ uniformly in $C$.
\end{proof}

Similar results were previously obtained by Feuer and Heymann \cite[Theorem~3.3]{FeHe-1976}
and by Heymann and Stern in
\cite[Theorem~2.1]{HeSt-1976} for the autonomous control system
$\dot{y} = f(y,w)$ (with $w\in \mathbb{R}^{m}$ a control parameter),
using fixed point theorems for multivalued mappings by
Kakutani \cite{Ka-1941} and
Browder \cite{Br-1968}, respectively. However, we point out that
in \cite{FeHe-1976,HeSt-1976}
it is assumed that $f=f(y,w)$ is continuously differentiable in
its first variable while, in our case, only the continuity of $f=f(y)$
is required. Moreover, in our approach, as well as in \cite{CiMa-2021},
the results are obtained by a simple application of the Brouwer fixed point theorem,
without invoking more advanced fixed point theorems on multivalued mappings.
Incidentally, Theorem~\ref{th-cm1} is also connected to a classical
result by Gustafson and Schmitt \cite[Theorem~1]{GuSc-1974}
(for non-autonomous delay-differential equations),
which was stated with strict inequalities in \eqref{eq-cm}. Recent
extensions to first order systems with nonlocal boundary conditions
have been obtained by Mawhin and Szyma\'{n}ska-D\k{e}bowska in \cite{MaSz-2017},
see also \cite[Theorem~1.2]{MaSz-2017b} and \cite[Theorem~2.12]{MaSz-2019}.
These results exploit the topological degree theory.

\begin{remark}\label{rem-2.2}
Theorem~\ref{th-cm} and Theorem~\ref{th-cm1}, although quite similar in the assumptions, are independent. Indeed, if the interior of the convex set $C$ is empty, we cannot replace the condition \textit{for each outer normal field} with \textit{there exists an outer normal field}, as the following elementary planar example shows.

Consider the system $\dot{x}_{1} = 1$, $\dot{x}_{2} = -x_{2}$ (without equilibrium points) and the convex set
$C:=\mathopen{[}1,2\mathclose{]} \times \{0\}\subset \mathbb{R}^{2}$. Notice that $\partial C= C$. For each $x=(x_{1},0)\in \partial C= C$, we take $\nu(x)=(0,1)$ and observe that $\partial C\ni x\mapsto  \nu(x) \in \mathbb{S}^{1}$ is a normal outer field for the set $C$, because, $\langle u - x, \nu(x) \rangle \leq 0$ $(=0)$ for all $x\in \partial C$ and $u\in C$. Moreover, \eqref{eq-cm} is satisfied (the inner product is identically zero) for this particular normal outer field and the map $f \colon (x_{1},x_{2})\mapsto (1,-x_{2})$.

Another (even simpler) example was suggested by Prof.~Mawhin (private communication): consider $f \colon (x_{1},x_{2})\mapsto (1,0)$, $C=\mathopen{[}-1,1\mathclose{]}\times\{0\}$ and $\nu(x)=(0,1)$ at every point of $C=\partial C$.
\hfill$\lhd$
\end{remark}

\section{Positively invariant sets and geometric conditions at the boundary}\label{section-3}

Throughout this section we suppose that $f=f(t,x) \colon \mathopen{[}0,T\mathclose{]}\times \Omega \to \mathbb{R}^{N}$ is a continuous (non-autonomous) vector field with $\Omega \subset \mathbb{R}^{N}$ a nonempty open set and $G$ is an open and bounded set with $\overline{G}\subset \Omega$. Unless explicitly stated, we do not assume the uniqueness of the solutions for initial value problems associated with
\begin{equation}\label{eq-3.1}
\dot{y} = f(t,y).
\end{equation}
When we consider a solution of \eqref{eq-3.1}, we implicitly assume that it is \textit{non-continuable}, i.e.~defined on a maximal interval of existence. Our aim is to prove the existence of a $T$-periodic solution of \eqref{eq-3.1} in $\overline{G}$, namely a solution satisfying the $T$-periodic boundary condition $y(0) = y(T)$, and such that $y(t)\in \overline{G}$ for all $t\in \mathopen{[}0,T\mathclose{]}$. When this kind of solutions exists, then in the particular case when the vector field is autonomous, we have also equilibrium points in the set $\overline{G}$. This is essentially Lemma~\ref{lem-2.1} in the context of the autonomous system \eqref{eq-2.1} and without the assumption of uniqueness for the Cauchy problems. Indeed, we have the following.

\begin{lemma}\label{lem-3.1}
Let $f \colon \Omega \to \mathbb{R}^{N}$ be a continuous vector field. Assume that there exists a sequence of $T_{k}$-periodic solutions to \eqref{eq-2.1} with values in $\overline{G}$ and with $T_{k}\searrow 0$. Then, there exists $z^{*}\in \overline{G}$ with $f(z^{*})=0$.
\end{lemma}

\begin{proof}
A proof based on Ascoli--Arzel\`{a} theorem can be found in \cite[Section~4]{CiMa-2021}. We provide an alternative one as follows. Let $\zeta_{k}$ be a sequence of $T_{k}$-periodic solutions of \eqref{eq-2.1} with $\zeta_{k}(t)\in \overline{G}$ for all $t$. By the compactness of $\overline{G}$ we can suppose, without loss of generality, that there exists a point $z^{*}\in \overline{G}$ such that
$\zeta_{k}(0)\to z^{*}$. Observe also that $\|\zeta_{k}(t)-\zeta_{k}(0)\|\leq T_{k} K$ for all $t\in \mathopen{[}0,T_{k}\mathclose{]}$, where $K\geq \max\{\|f(w)\| \colon w\in \overline{G}\}$. Next, we observe that, for any fixed vector $\nu\in \mathbb{R}^{N}$, it holds that
\begin{equation*}
0 = \frac{1}{T_{k}}\langle \zeta_{k}(T_{k})-\zeta_{k}(0), \nu\rangle  =
\frac{1}{T_{k}} \int_{0}^{T_{k}} \langle f(\zeta_{k}(t)), \nu\rangle \,\mathrm{d}t
= \langle f(\zeta_{k}(\tilde{t}_{k})), \nu\rangle,
\end{equation*}
where in the last equality we have applied the mean value theorem.
From $\|\zeta_{k}(\tilde{t}_{k}) - z^{*}\|\leq T_{k} K + \|\zeta_{k}(0) - z^{*}\|\to 0$ and the continuity of $f$, we find that $\langle f(z^{*}),\nu\rangle = 0$ and, finally, we conclude that $f(z^{*})=0$, since the nullity of the inner product holds for an arbitrary vector $\nu$.
\end{proof}

The existence of $T$-periodic solutions with values in $\overline{G}$ for the non-autonomous system \eqref{eq-3.1} will
be proved as a consequence of the Brouwer fixed point theorem and some geometric conditions at the boundary of $G$ which are related (but not included) to some analogous ones considered in the theory of positively invariant sets.
In more detail, following \cite{FeZa-1988} and the theory of \textit{bound sets} \cite{GaMa-1977, Ma-1979}, we consider open and bounded sets whose boundary if
determined by a family of \textit{bounding functions} $(V_{u})_{u\in \partial G}$
with the following property: for each $u\in \partial G$ there exist a radius $r_{u}>0$
and $\mathcal{C}^{1}$-function $V_{u} \colon B(u,r_{u})\to \mathbb{R}$ such that $V_{u}(u)=0$
and $G\cap B(u,r_{u}) \subset \{x\in B(u,r_{u}) \colon V_{u}(x)< 0\}$.
This definition is a slight modification of \cite[Definition~2.1]{FeZa-2021}. For simplicity of the exposition,
throughout the paper we confine ourselves to the case of smooth (i.e.~$\mathcal{C}^{1}$) bounding functions.
The theory can be extended to the case of locally Lipschitz functions, too (cf.~\cite{Ta-2000,Za-1987TS}), but we do not consider here this case which is beyond the scope of this article.

The next theorem summarizes some results previously obtained in \cite{Ma-1977}
(for the \textit{attractive bound sets}) and in \cite{FeZa-1988}
(for \textit{strongly flow-invariant sets}).

\begin{theorem}\label{th-3.1}
Suppose that there exists a family $(V_{u})_{u\in \partial G}$ of bounding functions for $G$ such that for each $u\in \partial G$ there exists $\rho_{u} \in\mathopen{]}0, r_{u}\mathclose{]}$ such that
\begin{equation}\label{eq-3.2}
\langle f(t,x), \nabla V_{u}(x)\rangle \leq 0,
\quad \text{for all $x\in G\cap B(u,\rho_{u})$ and $t\in \mathopen{[}0,T\mathclose{]}$.}
\end{equation}
Then,
\begin{itemize}
\item[$(i)$] each solution of \eqref{eq-3.1} with $y(0)\in G$ is defined on $\mathopen{[}0,T\mathclose{]}$
and such that $y(t)\in G$ for all $t\in \mathopen{[}0,T\mathclose{]}$
(\textit{strong flow invariance for $G$});
\item[$(ii)$] for each $x_{0}\in \partial G$, there exists at least one solution of \eqref{eq-3.1} defined on $\mathopen{[}0,T\mathclose{]}$ such that $y(0)=x_{0}$ and $y(t)\in \overline{G}$ for all $t\in \mathopen{[}0,T\mathclose{]}$
(\textit{weak flow invariance for $\overline{G}$});
\item[$(iii)$] if $\overline{G}$ is homeomorphic to a closed ball of $\mathbb{R}^{N}$ and we assume the uniqueness for the Cauchy problems associated with \eqref{eq-3.1} with initial value in $\overline{G}$, then system \eqref{eq-3.1} has a $T$-period solution in $\overline{G}$.
\end{itemize}
\end{theorem}

\begin{proof}
Let $y$ be an arbitrary solution of \eqref{eq-3.1}, with $y(0)\in G$. We claim that the solution is defined on the whole
interval $\mathopen{[}0,T\mathclose{]}$ and, moreover, $y(t)\in G$, for all $t\in \mathopen{[}0,T\mathclose{]}$.
If, by contradiction, this is not true, then there exists a time $\hat{t}\in \mathopen{]}0,T\mathclose{]}$ such that $y(t)\in G$, for all $t\in \mathopen{[}0,\hat{t}\mathclose{[}$ and $y(\hat{t})=:u\in \partial G$. Let $\varepsilon\in\mathopen{]}0,\hat{t}\mathclose{[}$
be such that $y(t)\in G\cap B(u,r_{u})$, for all $t\in \mathopen{[} \hat{t}- \varepsilon,\hat{t} \mathclose{[}$, and consider the auxiliary function $\eta(t):=V_{u}(y(t))$ for $t\in \mathopen{[}\hat{t}- \varepsilon,\hat{t}\mathclose{]}$. We observe that $\eta(t)< 0$ for all $t\in [\hat{t}- \varepsilon,\hat{t}[$ and $\eta(\hat{t})=0$.
On the other hand, for all $t\in \mathopen{[}\hat{t}- \varepsilon,\hat{t}\mathclose{]}$, it holds that $\eta'(t) = \langle \dot{y}(t), \nabla V_{u}(y(t))\rangle = \langle f(t,y(t)), \nabla V_{u}(y(t))\rangle \leq 0$, according to \eqref{eq-3.2}, so that $0= \eta(\hat{t}) \leq \eta(\hat{t}- \varepsilon) < 0$ and
a contradiction is obtained.

Having proved $(i)$, we immediately obtain $(ii)$ as follows.
If $u\in \partial G$, then there exists a sequence of points $z_{n}\in G$, with $z_{n}\to u$. For each $z_{n}$, let
$y_{n}$ be any solution of \eqref{eq-3.1} with $y_{n}(0)=z_{n}$, with $y_{n}$ defined on $\mathopen{[}0,T\mathclose{]}$ and such that $y_{n}(t)\in G$, for all $t\in \mathopen{[}0,T\mathclose{]}$. Then \cite[Theorem~7.2]{Ha-1972} ensures the existence of a subsequence $(y_{k_{n}})_{n}$ of $(y_{n})_{n}$ with $(y_{k_{n}})_{n}$ converging (uniformly on $\mathopen{[}0,T\mathclose{]}$) to a solution $\tilde{y}$ of \eqref{eq-3.1} with $\tilde{y}$ defined on $\mathopen{[}0,T\mathclose{]}$. We conclude that $\tilde{y}(t)\in \overline{G}$ for all $t\in \mathopen{[}0,T\mathclose{]}$, because $y_{k_{n}}(t)\to \tilde{y}(t)$, with $y_{k_{n}}(t)\in G$.

If we further assume the uniqueness of the solutions for the Cauchy problems associated with \eqref{eq-3.1} with initial value in $\overline{G}$,
we have that the Poincar\'{e} map $\Phi_{0}^{T}$ is well defined on $\overline{G}$ and maps $\overline{G}$ into itself (as a consequence of $(i)$ and $(ii)$). The Brouwer fixed point theorem guarantees the existence
of a fixed point $z^{*}\in \overline{G}$ for $\Phi_{0}^{T}$
and hence the existence of a $T$-periodic solution $\hat{y}$ of \eqref{eq-3.1}
with $\hat{y}(0)=z^{*}$ and such that $\hat{y}(t)\in \overline{G}$ for all $t\in \mathopen{[}0,T\mathclose{]}$ (by $(i)$ and $(ii)$). Thus also $(iii)$ is proved.
\end{proof}

Clearly, a sufficient condition for the validity of \eqref{eq-3.2} is to assume
\begin{equation}\label{eq-3.3}
\langle f(t,u), \nabla V_{u}(u)\rangle < 0,
\quad \text{for all $u\in \partial G$ and $t\in \mathopen{[}0,T\mathclose{]}$,}
\end{equation}
(see \cite[Theorem~7.4]{Ma-1977}). In this case, we just restrict
each $V_{u}$ to a smaller ball $B(u,r'_{u})\subset B(u,r_{u})$
(if necessary), and \eqref{eq-3.2} is satisfied.

In general, we cannot relax condition \eqref{eq-3.3} to
\begin{equation}\label{eq-3.4}
\langle f(t,u), \nabla V_{u}(u)\rangle \leq 0,
\quad \text{for all $u\in \partial G$ and $t\in \mathopen{[}0,T\mathclose{]}$,}
\end{equation}
even if we require the uniqueness of the solutions for the initial value problems,
as it can be seen by the following elementary example.

\begin{example}\label{ex-3.1}
Consider the one-dimensional system $\dot{y}=1$
and the set $G:=\{x \in \mathbb{R} \colon V(x)<0 \}$, with $V(x):= (x+1)(x-1)^{3}$.
We have $G=\mathopen{]}-1,1\mathclose{[}$, and $V_{-1} = V_{1} = V$,
with \eqref{eq-3.4} satisfied.
All the conclusions in Theorem~\ref{th-3.1} fail. The problem here is due to the fact that $\nabla V_{u}(u)=0$
at the point $u=1\in \partial G$.
\hfill$\lhd$
\end{example}

\subsection{A general existence result for $T$-periodic solutions}\label{section-3.1}

In view of Example~\ref{ex-3.1}, a natural assumption to improve the hypotheses \eqref{eq-3.2} or \eqref{eq-3.3} to \eqref{eq-3.4} is to add the condition
\begin{equation}\label{hp-non-deg}
\nabla V_{u}(u)\neq 0, \quad \text{for all $u\in \partial G$.}
\end{equation}
Actually, it will be convenient to suppose the following condition: there are constants $0 < \eta^- \leq \eta^+$ such that
\begin{equation}\label{hp-non-deg*}
\eta^-\leq \|\nabla V_{u}(u)\|\leq \eta^+, \quad \text{for all $u\in \partial G$.}
\end{equation}
When assumption \eqref{hp-non-deg*}
is required, we say that $(V_{u})_{u\in\partial G}$ is a
family of \textit{non-degenerate} bounding functions.
Observe that if \eqref{hp-non-deg} is satisfied, we can always enter in condition \eqref{hp-non-deg*}, by passing to the new family $(W_{u})_{u\in\partial G}$ with $W_{u}(x):= V_{u}(x)/\|\nabla V_{u}(u)\|$. For the new bounding functions, condition \eqref{hp-non-deg*} is satisfied with $\eta^- = \eta^+ =1$.

Our main result in this direction is the following.

\begin{theorem}\label{th-main}
Let $(V_{u})_{u\in \partial G}$ be a family of non-degenerate bounding functions for $G$ such that \eqref{eq-3.4}
holds and suppose that there exists a continuous function
$g \colon \partial{G}\to \mathbb{R}^{N}$ such that
\begin{equation}\label{cond-th-main}
\sup_{u\in \partial G}\langle g(u), \nabla V_{u}(u)\rangle < 0.
\end{equation}
If $\overline{G}$ is homeomorphic to a closed ball of $\mathbb{R}^{N}$, then there exists a $T$-periodic solution of \eqref{eq-3.1} with values in $\overline{G}$.
\end{theorem}

\begin{proof}
Using Tietze extension theorem we extend by continuity $g$ to $\overline{G}$.
Let $\varepsilon >0$ be such that $\langle g(u), \nabla V_{u}(u)\rangle \leq - \varepsilon$, for every $u\in
\partial G$. Then, for every positive integer $n$, we consider a locally Lipschitz continuous
function $f_{n} \colon \Omega\times \mathopen{[}0,1\mathclose{]} \to \mathbb{R}^{N}$ such that
\begin{equation*}
\Bigl{\|}f_{n}(t,x) - (f(t,x) + \frac{1}{n}g(x))\Bigr{\|}\leq \frac{\varepsilon}{2n \eta^+},
\quad \text{for every $(t,x)\in \mathopen{[}0,T\mathclose{]}\times \overline{G}$.}
\end{equation*}
We have that
\begin{equation*}
\|f_{n}(t,x) - f(t,x)\| \leq \frac{\varepsilon}{2n\eta^+} +
\frac{1}{n} \max_{x\in \overline{G}} \|g(x)\|,
\quad \text{for every $(t,x)\in \mathopen{[}0,T\mathclose{]}\times \overline{G}$,}
\end{equation*}
so that $f_{n}$ converges uniformly to $f$ on $\mathopen{[}0,T\mathclose{]}\times \overline{G}$.
Moreover, for all $u\in \partial G$ and $t\in \mathopen{[}0,T\mathclose{]}$, we deduce that
\begin{align*}
&\langle f_{n}(t,u), \nabla V_{u}(u)\rangle =
\\
&= \langle f_{n}(t,u) - (f(t,u) + \frac{1}{n} g(u)),\nabla V_{u}(u)\rangle
+ \langle f(t,u) + \frac{1}{n} g(u),\nabla V_{u}(u)\rangle
\\
&\leq \Bigl{\|}f_{n}(t,u) - (f(t,u) + \frac{1}{n} g(u))\Bigr{\|} \|\nabla V_{u}(u)\|
+ \langle f(t,u), \nabla V_{u}(u)\rangle + \frac{1}{n}\langle g(u), \nabla V_{u}(u)\rangle
\\
&\leq \frac{\varepsilon}{2n} + 0 + \frac{-\varepsilon}{n} < 0,
\end{align*}
so that condition \eqref{eq-3.3} is satisfied for $f_{n}$. Hence, we can apply $(iii)$ of Theorem~\ref{th-3.1} to the differential systems
\begin{equation}\label{eq-3.1n}
\dot{y}= f_{n}(t,y)
\end{equation}
and obtain that, for each $n$, there exists a $T$-periodic solution $\tilde{y}_{n}$ to \eqref{eq-3.1n} such that $\tilde{y}_{n}(t)\in \overline{G}$ for all $t\in \mathopen{[}0,T\mathclose{]}$. Using the Ascoli--Arzel\`{a} theorem and the fact that
$f_{n}$ converges uniformly to $f$ on $\mathopen{[}0,T\mathclose{]}\times \overline{G}$, we can conclude that there exists a subsequence $(\tilde{y}_{k_{n}})_{n}$ of $(\tilde{y}_{n})_{n}$ converging uniformly on $\mathopen{[}0,T\mathclose{]}$ to a $T$-periodic solution $\tilde{y}$ of \eqref{eq-3.1} and such that $\tilde{y}(t) \in \overline{G}$ for all $t\in \mathopen{[}0,T\mathclose{]}$.
\end{proof}

We stress the fact that in Theorem~\ref{th-main} we do not require the uniqueness of the solutions for the Cauchy problems. It is immediate to check that Theorem~\ref{th-cm1} can be obtained as a consequence of Theorem~\ref{th-main}, setting $V_{u}(x):= \langle x-u,\nu(u)\rangle$ and $g(x):=w_{0}-x$ with $w_{0}\in \mathrm{int\,}C$ (and using also Lemma~\ref{lem-3.1}).

Another application, following \cite[Theorem~16.9, p.~218]{Am-1990}, can be provided when $V_{u}\equiv V$ for all $u\in \partial G$, namely, when $\overline{G}$ is a sub-level set of a Lyapunov-like function. More precisely, let us suppose that $V \colon\Omega\to \mathbb{R}$ is a continuously differentiable function and let $c\in \mathbb{R}$ be such that $M:=V^{-1}(\mathopen{]}-\infty,c\mathclose{]})$ is a compact set with nonempty interior.
Suppose also that
\begin{equation}\label{eq-nondeg}
\nabla V(u)\neq0, \quad\text{for all } \, u\in \partial M.
\end{equation}
In this situation, $V^{-1}(\mathopen{]}-\infty,c\mathclose{[})\subset \mathrm{int\,}M$
and $\partial M\subset V^{-1}(c)$.
Moreover, $\partial M = \partial(\mathrm{int\,}M)$.
Hence, $G:= \mathrm{int\,}M$ is an open and bounded set
whose boundary is determined by the family of
bounding functions $(V_{u})_{u\in \partial G}$ such that $V_{u}(x):= V(x)-c$, with
$V_{u}$ restricted to an open ball $B(u,r_{u})$ where $\nabla V(x)\neq 0$
(see \cite[p.~520]{FeZa-2021}, for a similar discussion).
Having assumed that $\nabla V(x)\neq0$ for all $x\in \partial G=\partial M$, we have \eqref{hp-non-deg*} satisfied (by the continuity of $\nabla V$) and, moreover, \eqref{cond-th-main} holds for the choice $g(x)=-\nabla V(x)$.
Hence, the following corollary can be given.

\begin{corollary}\label{cor-3.1}
Let $V \colon \Omega\to \mathbb{R}$ be a $\mathcal{C}^{1}$-function and let $c\in \mathbb{R}$
be such that the set $M:=V^{-1}(\mathopen{]}-\infty,c\mathclose{]})$
is compact with nonempty interior.
Assume \eqref{eq-nondeg} and also
\begin{equation}\label{eq-corollary}
\langle f(t,u), \nabla V(u)\rangle \leq 0,
\quad \text{for all $u\in \partial M$ and $t\in \mathopen{[}0,T\mathclose{]}$.}
\end{equation}
If $M$ is homeomorphic to a closed ball of $\mathbb{R}^{N}$, then there exists a $T$-periodic solution of \eqref{eq-3.1} with values in $M$.
\end{corollary}

\begin{remark}\label{rem-3.1}
The conditions \eqref{eq-nondeg} and \eqref{eq-corollary}
are optimal. Indeed, concerning \eqref{eq-nondeg},
the same case considered in Example~\ref{ex-3.1} shows that the result is no more true if
$\nabla V$ vanishes at some point of the boundary. On the other hand, it is trivial
to produce cases where \eqref{eq-corollary} fails at some point of the boundary
and the conclusion of Corollary~\ref{cor-3.1} does not hold
(take, for instance, $f(t,y)=f(y)$ and $G=]-1,1[$ as in Example~\ref{ex-3.1} and
$V(x)=x^{2}-1$).
\hfill$\lhd$
\end{remark}

\begin{remark}\label{rem-3.2}
In order to apply Corollary~\ref{cor-3.1},
we should know the topological structure of the level
sets and sub-level sets of the Lyapunov-like functions.
This is a classical problem already studied in \cite{Wi-1967}
and also related to the Poincar\'{e} conjecture.
For a recent contribution, in the light of the verification of
the conjecture in all dimensions by Perelman, Freedman and Smale,
we refer to \cite{By-2008}.
See also \cite{MaPe-2006} and \cite[Corollary~11.2, p.~539]{Ha-1964}
for a connection to the Markus--Yamabe conjecture.
For instance, according to \cite[Theorem~1.2]{By-2008} we have that
the sub-level set $M = V^{-1}(]-\infty,c])$ is a compact set with nonempty
interior, if $V \colon \mathbb{R}^{N}\to \mathbb{R}$ is a smooth and proper function
with a compact set of critical points and $c$ is sufficiently large.
We also notice that Corollary~\ref{cor-3.1} and its consequences are
strongly related to the classical results of Krasnosel'ski\u{\i}
on guiding functions and the celebrated theorem
on the degree of the gradient of coercive maps \cite[Lemma~II.6.5]{Kr-1968}
(see also \cite{Am-1982,Am-1990} and \cite[Section~5.1.4]{DiMa-2021}).
\hfill$\lhd$
\end{remark}

\subsection{Applications to the non-convex case}\label{section-3.2}

Our aim now is to provide another application of Theorem~\ref{th-main}
outside the framework of convex sets.
With this respect, we deal with a class of star-shaped domains.
According to a standard terminology, a subset $\mathcal{U}$ of a vector space is said to be \textit{star-shaped}
with respect to a point $p\in \mathcal{U}$, if $\mathopen{[}p,x\mathclose{]}\subset \mathcal{U}$, for each $x\in \mathcal{U}$, where $\mathopen{[}p,x\mathclose{]} := \{p+\vartheta(x-p) \colon 0\leq \vartheta\leq 1\}$ is the segment connecting $p$ and $x$. Analogously, we define $\mathopen{[}p,x\mathclose{[} :=\{p+\vartheta(x-p) \colon 0\leq \vartheta<1\}$.

Usually, in the context of fixed point theory in Euclidean spaces, some further properties are required. In particular, if $\mathcal{U} = \overline{A}$ with $A\subset \mathbb{R}^{N}$ an open bounded set, we focus our study to the following case.

\begin{definition}\label{def-star}
Let $A\subset \mathbb{R}^{N}$ be a nonempty open bounded set. We say that $\overline{A}$ is \textit{strictly star-shaped} with respect to a point $p\in A$ if, for each point $u\in \partial A$,
we have $\mathopen{[}p,u\mathclose{[}\subset A$.
\end{definition}

Our definition corresponds to that considered by Deimling in \cite[p.~33]{De-1985}, referring to $\overline{A}$
as a star-shaped set with a simple boundary.
This is also the case presented by Yang in \cite[p.~111]{Ya-1999}, where it is required that, for every $u,v\in\partial A$ with $u\neq v$, it follows that $\mathopen{[}p,u\mathclose{]}\cap \mathopen{[}p,v\mathclose{]} = \{p\}$.
Clearly, compact convex sets with nonempty interior are strictly star-shaped (according to our definition) with respect to any interior point. Our definition requires a little more than the hypothesis that the open set $A$ is star-shaped. Indeed, if $A$ is star-shaped with respect to a point $p\in A$, also $\overline{A}$ is star-shaped with respect to $p$.
However, our condition on the boundary points might not be satisfied, as shown in the
following example.

\begin{example}\label{ex-3.2}
Let us consider the set
\begin{equation}\label{def-example}
A:=\Bigl{\{}(x_{1},x_{2})\in\mathbb{R}^{2}
\colon |x_{1}|^{\frac{1}{2}} + |x_{2}|^{\frac{1}{2}} <1 \Bigr{\}}
\setminus \Bigl{\{}(x_{1},x_{2})\in\mathbb{R}^{2}
\colon x_{1} \leq -\tfrac{1}{4}, x_{2}\geq 0 \Bigr{\}}.
\end{equation}
We notice that $A$
is star-shaped with respect to the origin as well as its closure
\begin{equation*}
\overline{A}= \Bigl{\{} (x_{1},x_{2})\in\mathbb{R}^{2}
\colon |x_{1}|^{\frac{1}{2}} + |x_{2}|^{\frac{1}{2}}
\leq 1 \Bigr{\}} \setminus \Bigl{\{} (x_{1},x_{2})
\in\mathbb{R}^{2} \colon x_{1} < -\tfrac{1}{4}, x_{2} >0 \Bigr{\}}.
\end{equation*}
Observe that the origin is the unique point $p$ such
that $\overline{A}$ star-shaped with respect to $p$.
On the other hand, $\overline{A}$ does non satisfy our condition.
Indeed, if $u=(u_{1},0)$ with $-1\leq u_{1} < -1/4$, then
$u\in \partial A$ but $\mathopen{[}0,u\mathclose{[}
\not\subset A$ and the condition of simplicity of the boundary of Deimling is not satisfied.
Equivalently, if $u=(u_{1},0)$ and $v=(v_{1},0)$ with $-1\leq u_{1}\neq v_{1} < -1/4$,
we find that $(-1/4,0)\in \mathopen{[}0,u\mathclose{]}\cap \mathopen{[}0,v\mathclose{]}$
and Yang's definition is not satisfied.
See Figure~\ref{fig-01} for a graphical representation of this example.
\hfill$\lhd$
\end{example}

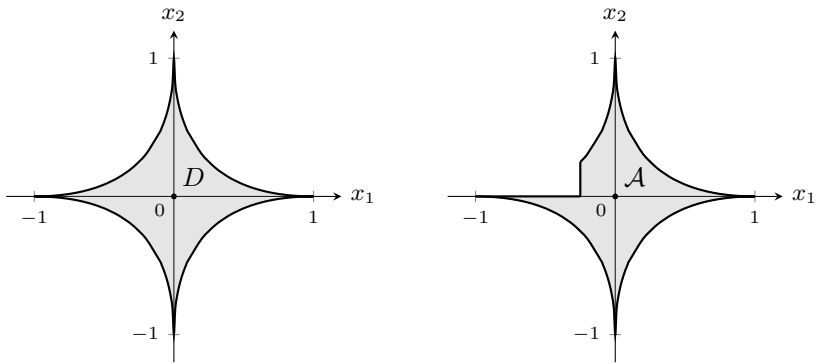
\begin{figure}[htb]
\centering
\begin{tikzpicture}
\begin{axis}[
  tick label style={font=\scriptsize},
  axis y line=middle,
  axis x line=middle,
  xtick={-1,1},
  ytick={-1,1},
  xticklabels={$-1$, $1$},
  yticklabels={$-1$, $1$},
  xlabel={\small $x_{1}$},
  ylabel={\small $x_{2}$},
every axis x label/.style={
    at={(ticklabel* cs:1.0)},
    anchor=west,
},
every axis y label/.style={
    at={(ticklabel* cs:1.0)},
    anchor=south,
},
  width=6cm,
  height=6cm,
  xmin=-1.2,
  xmax=1.2,
  ymin=-1.2,
  ymax=1.2]
\addplot [color=black, fill=gray, fill opacity=0.2, line width=0.8pt,smooth]coordinates {(-1, 0) (-0.9, 0.0026334) (-0.8, 0.0111456) (-0.7, 0.0266799) (-0.6, 0.0508067) (-0.5, 0.0857864) (-0.4, 0.135089) (-0.3, 0.204555) (-0.2, 0.305573) (-0.1, 0.467544) (-0.09, 0.49) (-0.08, 0.514315) (-0.07, 0.54085) (-0.06, 0.570102) (-0.05, 0.602786) (-0.04, 0.64) (-0.03, 0.68359) (-0.02, 0.737157) (-0.01, 0.81) (0, 1) (0.01, 0.81) (0.02, 0.737157) (0.03, 0.68359) (0.04, 0.64) (0.05, 0.602786) (0.06, 0.570102) (0.07, 0.54085) (0.08, 0.514315) (0.09, 0.49) (0.1, 0.467544) (0.2, 0.305573) (0.3, 0.204555) (0.4, 0.135089) (0.5, 0.0857864) (0.6, 0.0508067) (0.7, 0.0266799) (0.8, 0.0111456) (0.9, 0.0026334) (1, 0) (0.9, -0.0026334) (0.8, -0.0111456) (0.7, -0.0266799) (0.6, -0.0508067) (0.5, -0.0857864) (0.4, -0.135089) (0.3, -0.204555) (0.2, -0.305573) (0.1, -0.467544) (0.09, -0.49) (0.08, -0.514315) (0.07, -0.54085) (0.06, -0.570102) (0.05, -0.602786) (0.04, -0.64) (0.03, -0.68359) (0.02, -0.737157) (0.01, -0.81) (0, -1) (-0.01, -0.81) (-0.02, -0.737157) (-0.03, -0.68359) (-0.04, -0.64) (-0.05, -0.602786) (-0.06, -0.570102) (-0.07, -0.54085) (-0.08, -0.514315) (-0.09, -0.49) (-0.1, -0.467544) (-0.2, -0.305573) (-0.3, -0.204555) (-0.4, -0.135089) (-0.5, -0.0857864) (-0.6, -0.0508067) (-0.7, -0.0266799) (-0.8, -0.0111456) (-0.9, -0.0026334) (-1, 0)};
\node at (axis cs: -0.1,-0.1) {\scriptsize{$0$}};
\node at (axis cs: 0.14,0.14) {$D$};
\fill (axis cs: 0,0) circle (1.1pt);
\end{axis}
\end{tikzpicture}
\quad\quad
\begin{tikzpicture}
\begin{axis}[
  tick label style={font=\scriptsize},
  axis y line=middle,
  axis x line=middle,
  xtick={-1,1},
  ytick={-1,1},
  xticklabels={$-1$, $1$},
  yticklabels={$-1$, $1$},
  xlabel={\small $x_{1}$},
  ylabel={\small $x_{2}$},
every axis x label/.style={
    at={(ticklabel* cs:1.0)},
    anchor=west,
},
every axis y label/.style={
    at={(ticklabel* cs:1.0)},
    anchor=south,
},
  width=6cm,
  height=6cm,
  xmin=-1.2,
  xmax=1.2,
  ymin=-1.2,
  ymax=1.2]
\addplot [color=black, fill=gray, fill opacity=0.2, line width=0pt] coordinates {(-1, 0) (-0.25,0) (-0.25, 0.25) (-0.2, 0.305573) (-0.1, 0.467544) (-0.09, 0.49) (-0.08, 0.514315) (-0.07, 0.54085) (-0.06, 0.570102) (-0.05, 0.602786) (-0.04, 0.64) (-0.03, 0.68359) (-0.02, 0.737157) (-0.01, 0.81) (0, 1) (0.01, 0.81) (0.02, 0.737157) (0.03, 0.68359) (0.04, 0.64) (0.05, 0.602786) (0.06, 0.570102) (0.07, 0.54085) (0.08, 0.514315) (0.09, 0.49) (0.1, 0.467544) (0.2, 0.305573) (0.3, 0.204555) (0.4, 0.135089) (0.5, 0.0857864) (0.6, 0.0508067) (0.7, 0.0266799) (0.8, 0.0111456) (0.9, 0.0026334) (1, 0) (0.9, -0.0026334) (0.8, -0.0111456) (0.7, -0.0266799) (0.6, -0.0508067) (0.5, -0.0857864) (0.4, -0.135089) (0.3, -0.204555) (0.2, -0.305573) (0.1, -0.467544) (0.09, -0.49) (0.08, -0.514315) (0.07, -0.54085) (0.06, -0.570102) (0.05, -0.602786) (0.04, -0.64) (0.03, -0.68359) (0.02, -0.737157) (0.01, -0.81) (0, -1) (-0.01, -0.81) (-0.02, -0.737157) (-0.03, -0.68359) (-0.04, -0.64) (-0.05, -0.602786) (-0.06, -0.570102) (-0.07, -0.54085) (-0.08, -0.514315) (-0.09, -0.49) (-0.1, -0.467544) (-0.2, -0.305573) (-0.3, -0.204555) (-0.4, -0.135089) (-0.5, -0.0857864) (-0.6, -0.0508067) (-0.7, -0.0266799) (-0.8, -0.0111456) (-0.9, -0.0026334) (-1, 0)};
\addplot [color=black, line width=0.8pt] coordinates {(-1, 0) (-0.25,0)};
\addplot [color=black, line width=0.8pt, smooth] coordinates {(-0.25, 0.25) (-0.2, 0.305573) (-0.1, 0.467544) (-0.09, 0.49) (-0.08, 0.514315) (-0.07, 0.54085) (-0.06, 0.570102) (-0.05, 0.602786) (-0.04, 0.64) (-0.03, 0.68359) (-0.02, 0.737157) (-0.01, 0.81) (0, 1) (0.01, 0.81) (0.02, 0.737157) (0.03, 0.68359) (0.04, 0.64) (0.05, 0.602786) (0.06, 0.570102) (0.07, 0.54085) (0.08, 0.514315) (0.09, 0.49) (0.1, 0.467544) (0.2, 0.305573) (0.3, 0.204555) (0.4, 0.135089) (0.5, 0.0857864) (0.6, 0.0508067) (0.7, 0.0266799) (0.8, 0.0111456) (0.9, 0.0026334) (1, 0) (0.9, -0.0026334) (0.8, -0.0111456) (0.7, -0.0266799) (0.6, -0.0508067) (0.5, -0.0857864) (0.4, -0.135089) (0.3, -0.204555) (0.2, -0.305573) (0.1, -0.467544) (0.09, -0.49) (0.08, -0.514315) (0.07, -0.54085) (0.06, -0.570102) (0.05, -0.602786) (0.04, -0.64) (0.03, -0.68359) (0.02, -0.737157) (0.01, -0.81) (0, -1) (-0.01, -0.81) (-0.02, -0.737157) (-0.03, -0.68359) (-0.04, -0.64) (-0.05, -0.602786) (-0.06, -0.570102) (-0.07, -0.54085) (-0.08, -0.514315) (-0.09, -0.49) (-0.1, -0.467544) (-0.2, -0.305573) (-0.3, -0.204555) (-0.4, -0.135089) (-0.5, -0.0857864) (-0.6, -0.0508067) (-0.7, -0.0266799) (-0.8, -0.0111456) (-0.9, -0.0026334) (-1, 0)};
\addplot [color=black, line width=0.8pt] coordinates {(-0.25,0) (-0.25, 0.25)};
\node at (axis cs: -0.1,-0.1) {\scriptsize{$0$}};
\node at (axis cs: 0.14,0.14) {$\mathcal{A}$};
\fill (axis cs: 0,0) circle (1.1pt);
\end{axis}
\end{tikzpicture}
\captionof{figure}{Representation of the set
$D:=\{(x_{1},x_{2})\in\mathbb{R}^{2} \colon |x_{1}|^{\frac{1}{2}} + |x_{2}|^{\frac{1}{2}} <1 \}$
(on the left) and of the set $A$ defined in \eqref{def-example} (on the right).
Both $D$ and $A$ (and their closures) are star-shaped with respect to
the origin and the origin is the unique point such that these sets are
star-shaped with respect to it.
However, $\overline{D}$ is strictly star-shaped with
respect to the origin, while $\overline{A}$ is not strictly
star-shaped with respect to the origin (and not with respect to any other point).}
\label{fig-01}
\end{figure}

The class of star-shaped sets with the boundary condition considered in Definition~\ref{def-star}, is relevant in fixed
point theory. In particular, a continuous map $\phi$ such that
$\phi(\partial A)\subset \overline{A}$ has a fixed point
\cite[p.~33]{De-1985} (see also \cite[Corollary~1]{Za-1996}).
For other fixed point theorems in the setting of star-shaped sets, see \cite[Section~19.5]{HHMM-2020} and
\cite{Pa-2001}. Star-shaped sets with the stronger condition at
the boundary considered in Definition~\ref{def-star} are also
studied in different contexts. See \cite[p.~1008]{HHMM-2020} and the references therein
where these sets are also referred to as \textit{strongly star-shaped} or
\textit{radiative} at $p$. See also \cite{HaMa-2011} for a comparison of different properties concerning convex and
star-shaped sets.

In the next lemma we show that a strictly star-shaped set is homeomorphic to a closed unit ball
(since we have not found a precise reference in the literature, we give also the simple proof,
for the reader's convenience).

\begin{lemma}\label{lem-3.2}
Let $A\subset \mathbb{R}^{N}$ be a nonempty open bounded set with $\overline{A}$ strictly star-shaped with respect to a point $p\in A$. Then, for each unit vector $v\in \mathbb{S}^{N-1}$ there is a unique $\vartheta_{v}\in \mathopen{]}0,+\infty\mathclose{[}$ such that $p+\vartheta_{v} v \in \partial A$. Moreover, the map $\mathbb{S}^{N-1}\ni v\mapsto \vartheta_{v}$ is continuous and $\overline{A}$ is homeomorphic to the unit ball $B[0,1]$ of $\mathbb{R}^{N}$.
\end{lemma}

\begin{proof}
Let $0 < r < R$ be such that $B(p,r)\subset \overline{A}\subset B[p,R]$. Hence, for every $v\in \mathbb{S}^{N-1}$, the ray $\{p + t v \colon t\geq 0\}$ intersects $\partial A$ at a point $p+\vartheta_{v} v$ which is unique, as a consequence of Definition~\ref{def-star}. Observe that $r \leq \vartheta_{v} \leq R$ for all $v\in \mathbb{S}^{N-1}$. We claim that the map $\phi \colon \mathbb{S}^{N-1}\to \mathopen{[}r,R\mathclose{]}$ defined by $\phi(v):= \vartheta_{v}$ is continuous. Since the range of $\phi$ is contained in a compact set, it will be sufficient to prove that the graph of $\phi$ is closed (see \cite[Problem~108]{Wi-1983}). Accordingly, let $v_{n}\to v^{*}$ in $\mathbb{S}^{N-1}$ and $\vartheta_{v_{n}}\to \vartheta^{*}$. By definition, $p + \vartheta_{v_{n}} v_{n}\to p + \vartheta^{*} v^{*}$, with $p + \vartheta_{v_{n}} v_{n}\in \partial A$ for each $n$. Hence $p + \vartheta^{*} v^{*}\in \partial A$ and, by the uniqueness of the intersection of the ray $\{p + t v^{*} \colon t\geq 0\}$ with $\partial A$, we conclude that $\vartheta^{*}= \vartheta_{v^{*}}$. This proves that the graph of $\phi$ is closed.

Next, we introduce the map $\Psi \colon B[0,1] \to \overline{A}$, defined by
\begin{equation*}
\begin{cases}
\, \Psi(0)=p,
\\
\, \Psi(x)= p + x \phi\biggl{(}\dfrac{x}{\|x\|}\biggr{)}, \quad \text{for $x\neq 0$.}
\end{cases}
\end{equation*}
We have that $\Psi$ is continuous (by the continuity and the boundedness of $\phi$). Moreover, $\Psi|_{\mathbb{S}^{N-1}} \colon \mathbb{S}^{N-1}\to \partial A$, which maps $v\in \mathbb{S}^{N-1}$ to $p+\vartheta_{v} v \in \partial A$, is bijective and also $\Psi$ maps bijectively any segment $\mathopen{[}0,v\mathclose{]}\in B[0,1]$ to $\mathopen{[}p,p+\phi(v)v\mathclose{]}\in \overline{A}$. Consequently, $\Psi$ is a continuous one-to-one map from $B[0,1]$ onto $\overline{A}$ and hence a homeomorphism.
\end{proof}

The following property will be crucial in our result for non-convex sets.

\begin{lemma}\label{lem-3.3}
Let $G \subset \mathbb{R}^{N}$ be a nonempty open bounded set with $\overline{G}$ strictly star-shaped with respect to a point $p\in G$. Moreover, suppose that there exists a family $(V_{u})_{u\in \partial G}$ of bounding functions for $G$.
Then, it holds that
\begin{equation*}
\langle \nabla V_{u}(u),p-u\rangle \leq 0,
\quad \text{for all $u\in \partial G$.}
\end{equation*}
\end{lemma}

\begin{proof}
Let $u\in \partial G$ and $V_{u} \colon B(u,r_{u})\to \mathbb{R}$ be the corresponding bounding function. By definition, $V_{u}(u)=0$ and $V_{u}(x) < 0$ for all $x\in G\cap B(u,r_{u})$. Let us consider the segment $\mathopen{[}p,u\mathclose{]}\subset \overline{G}$ with $\mathopen{[}p,u\mathclose{[}\in G$. Taking a parametrization for the segment, we can introduce the function
$\gamma(\vartheta):=V_{u}(p + \vartheta(u-p))$, defined for $\vartheta \in \mathopen{]}1 - r_{u}/\|u-p\|,1\mathclose{]}$ and such that $\gamma(1) = 0$ and $\gamma(\vartheta)< 0$ for $\vartheta <1$. From this, we find that $\langle \nabla V_{u}(u),u-p\rangle = \gamma'(1) \geq 0$ and hence the thesis.
\end{proof}

After these preliminary results, we are now in position to give our application of Theorem~\ref{th-main} to star-shaped domains.

\begin{theorem}\label{th-starshaped}
Let $(V_{u})_{u\in \partial G}$ be a family of non-degenerate bounding functions
for a nonempty open bounded set $G$ with $\overline{G}$ strictly star-shaped with respect
to all the points in a ball $B[p,\delta]\subset G$.
Assume that \eqref{eq-3.4} holds.
Then, there exists a $T$-periodic solution of \eqref{eq-3.1} with values in $\overline{G}$.
\end{theorem}

\begin{proof}
First of all, without loss of generality, we suppose
(passing to $V_u(x)/\|\nabla V_{u}(u)\|$, if necessary)
that $\|\nabla V_{u}(u)\|=1$, for each $u\in \partial G$. Next, we observe that
\begin{equation*}
\langle \nabla V_{u}(u),p-u\rangle \leq -\delta,
\quad \text{for all $u\in \partial G$.}
\end{equation*}
Indeed, it is sufficient to apply Lemma~\ref{lem-3.3} to the point
$x:=p+\delta \, \nabla V_{u}(u)\in B[p,\delta]$,
so that, for each $u\in \partial G$, it holds that
\begin{equation*}
\langle \nabla V_{u}(u),p-u\rangle =
\langle \nabla V_{u}(u),x-u\rangle
-\delta \langle \nabla V_{u}(u),\nabla V_{u}(u)\rangle
\leq -\delta.
\end{equation*}
Now, defining $g \colon \overline{G}\to \mathbb{R}^{N}$ as $g(x)=p-x$,
we find that condition \eqref{cond-th-main} holds true.
Moreover, by Lemma~\ref{lem-3.2}, $\overline{G}$ is homeomorphic to the unit ball $B[0,1]$
of $\mathbb{R}^{N}$. Then, we reach the thesis as an application of Theorem~\ref{th-main}.
\end{proof}

\begin{remark}\label{rem-3.3}
The assumption that the set $\overline{G}$ is strictly star-shaped not only with respect to a point $p$,
but also with respect to all the points in a neighborhood of $p$,
is an hypothesis which is rather common in the theory of star-shaped
sets and is usually referred saying that the \textit{strong kernel}
of $\overline{G}$ has nonempty interior (cf.~\cite{HHMM-2020}).
Actually, if the open set $G$ has a \textit{kernel with nonempty interior}
(according to \cite[p.~1005]{HHMM-2020}), that is $G$
is star-shaped with respect to all the points of (small) open ball $B
\subset G$ then, according to \cite[Theorem~3, p.~1006]{HHMM-2020},
for each $p\in B$, and $u\in \partial G$, it follows that $\mathopen{[}p,u\mathclose{[}\subset G$.
As a consequence, $\partial \overline{G} = \partial G$ and $\overline{G}$ is strictly star-shaped with respect to $p$.
Therefore, $\overline{G}$ is strictly star-shaped
also with respect to all the points in a neighborhood of $p$.
\hfill$\lhd$
\end{remark}

\begin{example}\label{ex-3.3}
Let us consider the autonomous planar differential system
\begin{equation}\label{eq-nonuniq}
\dot{x_{1}} = 1, \quad \dot{x_{2}} = \varphi(x_{2})
\end{equation}
with
\begin{equation*}
\varphi(0)=0 \quad \text{and} \quad
\varphi(s):= -2 \dfrac{s}{\sqrt{|s|}}, \; \text{ for $s\neq 0$.}
\end{equation*}
Using the positions $x_{1}:=t$, $x_{2}:=x$, we see that
the positive semi-orbits of system \eqref{eq-nonuniq} are the graphs of
the solutions $(t,x(t))$ of $\dot{x} = \varphi(x)$ satisfying the initial condition
$x(t_0)= x_0$, for $t\geq t_0$ in the $(t,x)$-plane. The function $\varphi$
is continuous and decreasing, hence the forward uniqueness for the
solutions of the Cauchy problems is guaranteed, according to
\cite[Corollary~6.3, p.~34]{Ha-1964}.

We introduce now the set
\begin{equation}\label{def-G-ex-3.3}
G := \bigl{\{} (x_{1},x_{2}) \in \mathopen{]}-1,1\mathclose{[}
\times \mathopen{]}-1,1\mathclose{[} \colon |x_{2}| < (x_{1}-1)^{2} \bigr{\}}
\end{equation}
(see Figure~\ref{fig-02} for a graphical representation). The set $\overline{G}$ is strictly star-shaped only with respect to the points belonging
to the segment $\mathopen{[}-1,0\mathclose{]}\times\{0\}$
and thus is not strictly star-shaped with respect to all the points in any ball contained in $G$.

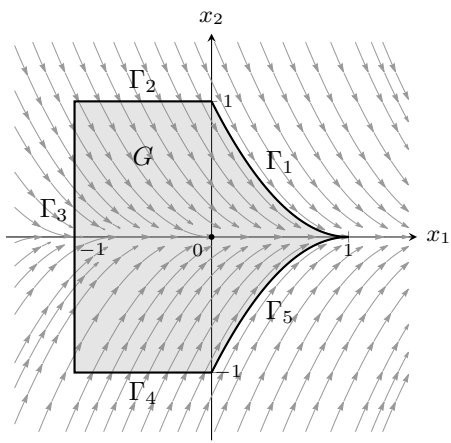
\begin{figure}[htb]
\centering
\begin{tikzpicture}
\begin{axis}[
  tick label style={font=\scriptsize},
  axis y line=middle,
  axis x line=middle,
  xtick={-1,1},
  ytick={-1,1},
  xticklabels={ , },
  yticklabels={ , },
  xlabel={\small $x_{1}$},
  ylabel={\small $x_{2}$},
every axis x label/.style={
    at={(ticklabel* cs:1.0)},
    anchor=west,
},
every axis y label/.style={
    at={(ticklabel* cs:1.0)},
    anchor=south,
},
  width=7cm,
  height=7cm,
  xmin=-1.5,
  xmax=1.5,
  ymin=-1.5,
  ymax=1.5]
\addplot[thick] graphics[xmin=-1.5,ymin=-1.5,xmax=1.5,ymax=1.5] {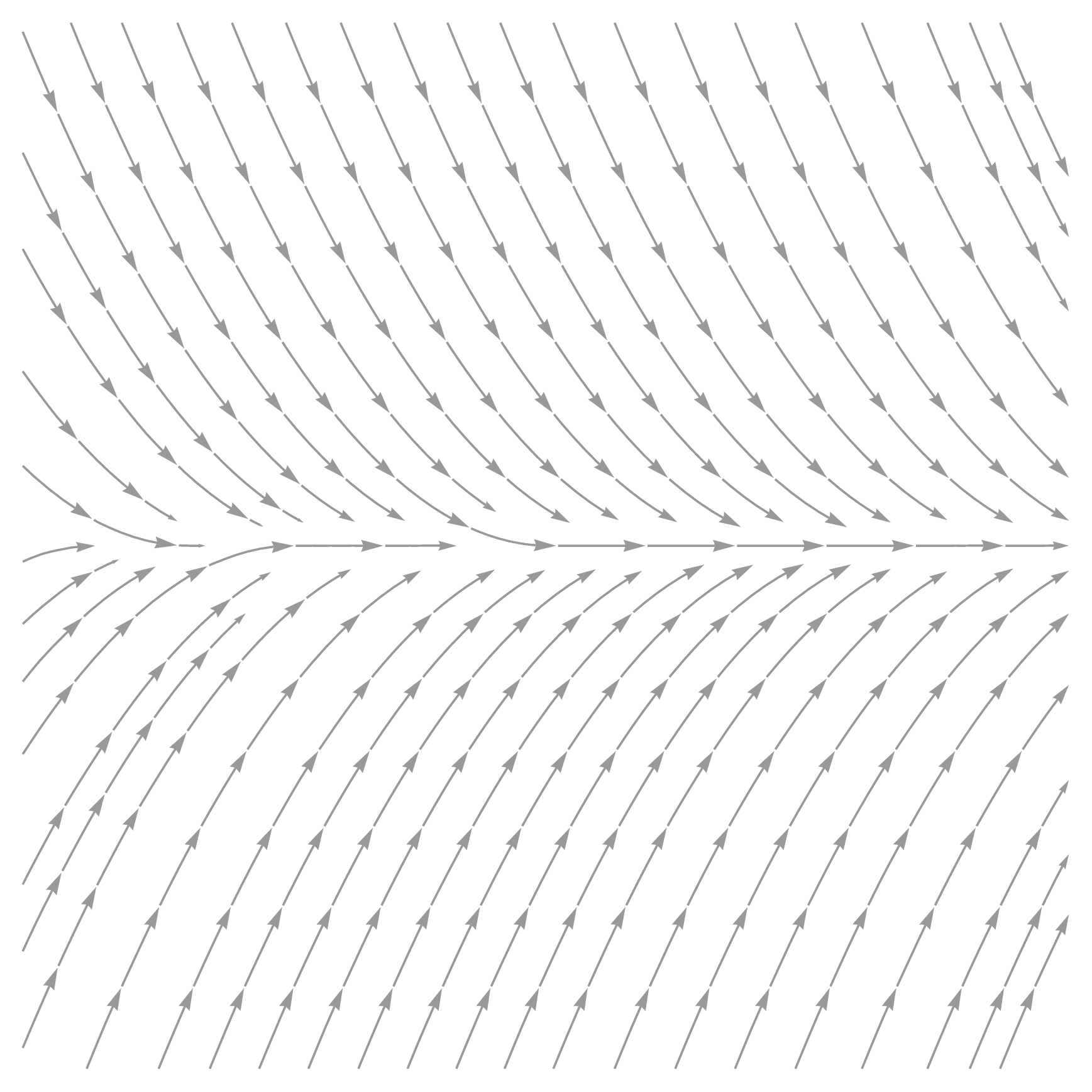};
\addplot [color=black, fill=gray, fill opacity=0.2, line width=0pt] coordinates {(0, -1) (-1, -1) (-1, 1) (0, 1) (0.05, 0.9025) (0.1, 0.81) (0.15, 0.7225) (0.2, 0.64) (0.25, 0.5625) (0.3, 0.49) (0.35, 0.4225) (0.4, 0.36) (0.45, 0.3025) (0.5, 0.25) (0.55, 0.2025) (0.6, 0.16) (0.65, 0.1225) (0.7, 0.09) (0.75, 0.0625) (0.8, 0.04) (0.85, 0.0225) (0.9, 0.01) (0.95, 0.0025) (1, 0) (0.95, -0.0025) (0.9, -0.01) (0.85, -0.0225) (0.8, -0.04) (0.75, -0.0625) (0.7, -0.09) (0.65, -0.1225) (0.6, -0.16) (0.55, -0.2025) (0.5, -0.25) (0.45, -0.3025) (0.4, -0.36) (0.35, -0.4225) (0.3, -0.49) (0.25, -0.5625) (0.2, -0.64) (0.15, -0.7225) (0.1, -0.81) (0.05, -0.9025) (0, -1)};
\addplot [color=black, line width=0.8pt] coordinates {(0, -1) (-1, -1) (-1, 1) (0, 1)};
\addplot [color=black, line width=0.8pt, smooth] coordinates {(0, 1) (0.05, 0.9025) (0.1, 0.81) (0.15, 0.7225) (0.2, 0.64) (0.25, 0.5625) (0.3, 0.49) (0.35, 0.4225) (0.4, 0.36) (0.45, 0.3025) (0.5, 0.25) (0.55, 0.2025) (0.6, 0.16) (0.65, 0.1225) (0.7, 0.09) (0.75, 0.0625) (0.8, 0.04) (0.85, 0.0225) (0.9, 0.01) (0.95, 0.0025) (1, 0) (0.95, -0.0025) (0.9, -0.01) (0.85, -0.0225) (0.8, -0.04) (0.75, -0.0625) (0.7, -0.09) (0.65, -0.1225) (0.6, -0.16) (0.55, -0.2025) (0.5, -0.25) (0.45, -0.3025) (0.4, -0.36) (0.35, -0.4225) (0.3, -0.49) (0.25, -0.5625) (0.2, -0.64) (0.15, -0.7225) (0.1, -0.81) (0.05, -0.9025) (0, -1)};
\node at (axis cs: -0.1,-0.1) {\scriptsize{$0$}};
\node at (axis cs: -0.5,0.6) {$G$};
\node at (axis cs: 0.5,0.55) {$\Gamma_{1}$};
\node at (axis cs: -0.5,1.15) {$\Gamma_{2}$};
\node at (axis cs: -1.15, 0.2) {$\Gamma_{3}$};
\node at (axis cs: -0.5,-1.15) {$\Gamma_{4}$};
\node at (axis cs: 0.5,-0.55) {$\Gamma_{5}$};
\node at (axis cs: 0.12,-1) {\scriptsize{$-1$}};
\node at (axis cs: 0.12,1) {\scriptsize{$1$}};
\node at (axis cs: -0.87,-0.1) {\scriptsize{$-1$}};
\node at (axis cs: 1,-0.1) {\scriptsize{$1$}};
\fill (axis cs: 0,0) circle (1.1pt);
\end{axis}
\end{tikzpicture}
\captionof{figure}{Representation of the set $G$ defined in \eqref{def-G-ex-3.3} and the flow associated with system \eqref{eq-nonuniq}.}
\label{fig-02}
\end{figure}

We split the boundary of $G$ into the following sets:
\begin{align*}
&\Gamma_{1}:= \bigl{\{} (x_{1},x_{2}) \colon 0\leq x_{1}\leq 1, x_{2}=(x_{1}-1)^{2} \bigr{\}},
\\
&\Gamma_{2}:= \bigl{\{} (x_{1},x_{2}) \colon -1\leq x_{1}\leq 0, x_{2}=1 \bigr{\}},
\\
&\Gamma_{3}:= \bigl{\{} (x_{1},x_{2}) \colon x_{1}=-1, -1\leq x_{2}\leq 1 \bigr{\}},
\\
&\Gamma_{4}:= \bigl{\{} (x_{1},x_{2}) \colon -1\leq x_{1}\leq 0, x_{2}=-1 \bigr{\}},
\\
&\Gamma_{5}:= \bigl{\{} (x_{1},x_{2}) \colon 0\leq x_{1}\leq 1, x_{2}= -(x_{1}-1)^{2} \bigr{\}}.
\end{align*}
A family of bounding functions associated with the set $G$ is given by
$(V_{u})_{u\in\partial G}$, with
\begin{align*}
&V_{u}(x_{1},x_{2})=V^{1}(x_{1},x_{2}):=- (x_{1}-1)^{2} + x_{2},
& \text{if $u\in \Gamma_{1}$,}
\\
&V_{u}(x_{1},x_{2})=V^{2}(x_{1},x_{2}):=-1+x_{2},
& \text{if $u\in \Gamma_{2}$,}
\\
&V_{u}(x_{1},x_{2})=V^{3}(x_{1},x_{2}):=-1-x_{1},
& \text{if $u\in \Gamma_{3}$,}
\\
&V_{u}(x_{1},x_{2})=V^{4}(x_{1},x_{2}):=-1-x_{2},
& \text{if $u\in \Gamma_{4}$,}
\\
&V_{u}(x_{1},x_{2})=V^{5}(x_{1},x_{2}):=- (x_{1}-1)^{2} - x_{2},
& \text{if $u\in \Gamma_{5}$.}
\end{align*}
We warn that, consistently with our definition of bounding functions, it is sufficient
to have \textit{one} function $V_{u}$ at any point $u\in \partial G$. For this reason,
when $u\in \Gamma_{i}\cap \Gamma_{j}$ with $i\neq j$, we will just choose one of the
two possibilities given by the above list.
Now, it is easy to check that $\langle f(u),\nabla V_{u}(u) \rangle = 0$
for all $u\in \Gamma_{1}\cup \Gamma_{5}$ and
$\langle f(u),\nabla V_{u}(u) \rangle < 0$ for all $u\in \Gamma_{2}\cup\Gamma_{3}\cup \Gamma_{4}$.
Hence condition \eqref{eq-3.4} is satisfied. However, there are no periodic solutions,
nor equilibrium points for system \eqref{eq-nonuniq}.
\hfill$\lhd$
\end{example}

\begin{remark}\label{rem-3.4}
Example~\ref{ex-3.3} shows that, if we assume the weak boundary condition
\eqref{eq-3.4}, then the condition of strong kernel with nonempty interior
of Theorem~\ref{th-starshaped} is optimal and cannot be removed.
Example~\ref{ex-3.3} also provides a new case which shows that the condition
of linear independence of the gradients considered in a result of
positively invariant sets in
\cite[Corollary~1, formula~(10)]{Ha-1972} cannot be removed.
Indeed, in \cite[Corollary~1]{Ha-1972}, Hartman
considered the case of a bounded domain whose boundary is described by a finite number
of bounding-type functions. Differently than in our case where,
at each point $u\in \partial G$, we take one function $V_{u}$, in \cite{Ha-1972}
there is a finite set of functions $V^{k}$ (denoted as $L^{k}$ in \cite{Ha-1972})
with the condition that the vectors $\nabla V^{k}(u)$ are linearly independent at the points
$u\in \partial G$ where the $V^{k}$'s vanish. In our example, at each point
of $\Gamma_{i}\cap \Gamma_{i+1}$, for $i=1,\dots,4$, the vectors $\nabla V^{i}(u)$ and
$\nabla V^{i+1}(u)$ are linearly independent, but this does not happen
at the point $z=(1,0)\in \Gamma_{1}\cap \Gamma_{5}$ where
$\nabla V^{1}(z)= (1,0) = - \nabla V^{5}(z)$.
Indeed, the set $\overline{G}$ is not positively invariant with respect to the semi-flow
associated with system~\eqref{eq-nonuniq}, because the solution
$(x_{1}(t),x_{2}(t))=(t,0)$ is in the set for $t=0$ and escapes the set $\overline{G}$
for $t>1$. The example of Hartman in \cite[p.~513]{Ha-1972} consider a trivial set
with empty interior, given by $\{(0,0)\}$, where the condition of independence of
the gradients fails.
\hfill$\lhd$
\end{remark}

A simple condition in order to verify that a sub-level set is strictly star-shaped
with a strong kernel with nonempty interior is provided by the next result.

\begin{lemma}\label{lem-3.4}
Let $V \colon \mathbb{R}^{N}\to \mathbb{R}$ (with $N\geq 2$)
be a continuously differentiable function
and let $p\in \mathbb{R}^{N}$ and $c\in \mathbb{R}$ be such that
$V(p) < c$ with $V^{-1}(c)$ bounded and nonempty. If
\begin{equation}\label{eq-rad}
\langle \nabla V(u),p-u\rangle < 0, \quad \text{for all $u\in V^{-1}(c)$,}
\end{equation}
then $M:=V^{-1}(\mathopen{]}-\infty,c\mathclose{]})$ is
strictly star-shaped with respect
to all the points in a ball $B[p,\delta]\subset G= \mathrm{int\,}M=
V^{-1}(\mathopen{]}-\infty,c\mathclose{[})$.
\end{lemma}

\begin{proof}
It is sufficient to check that condition \eqref{eq-rad} implies that the set
$M$ is strictly star-shaped with respect to $p$. Indeed, by the compactness of
$V^{-1}(c)$, we have $\max_{u\in V^{-1}(c)} \langle \nabla V(u),p-u\rangle < 0$
and hence, for all the points $z$ in a neighborhood of $p$, it holds that $\langle \nabla V(u), z-u\rangle < 0$, for all $u\in V^{-1}(c)$.

Now following (with some simplifications) an argument from \cite{Za-1996}, we consider the auxiliary function
$v(s):= V(p + s(u-p))$ for $u\in M\setminus \{p\}$ and $s \geq 0$.
By the assumptions, we deduce that $v(0)< c$, $v(1) \leq c$, and (from \eqref{eq-rad})
$v'(s) > 0$ for all $s$ such that $v(s)=c$. This proves that the half-line
$\{p + s(u-p) \colon s\geq 0\}$ intersects $\partial M = \partial G = V^{-1}(c)$
in at most one point.
More precisely, $\{p + s(u-p) \colon s\geq 0\}$ intersects $\partial G$
exactly in $u$ (for $s=1$), if $u\in\partial G$.

To conclude our proof, we have to show that, for every $v\in \mathbb{S}^{N-1}$, the
half-line $\{p + s v \colon s\geq 0\}$ intersects $V^{-1}(c)$.
Suppose, by contradiction, that there exists $\hat{v}\in \mathbb{S}^{N-1}$
such that $V(p + s \hat{v}) < c$, for all $s \geq 0$. We choose
a point $\hat{u}\in V^{-1}(c)$ such that
$\tfrac{\hat{u} - p}{\|\hat{u} - p\|} \neq -\hat{v}$.
Such a choice is always possible because, if $u\in V^{-1}(c)\neq \emptyset$
then $\nabla V(u)\neq 0$ by \eqref{eq-rad} and the implicit function theorem
guarantees that $V^{-1}(c)$ is locally a surface around $u$. Let $R> 0$ be such that
$V^{-1}(c)\subset B(p,R)$ and let $K>1$ be sufficiently large that $\mathopen{[}p + K \hat{v},p+ K(\hat{u}-p)\mathclose{]}\cap B(0,R)=\emptyset$.
On the other hand, $V(p + K \hat{v}) < c < V(p+ K(\hat{u}-p))$ and therefore
there exists $w\in \mathopen{]}p + K \hat{v},p+ K(\hat{u}-p)\mathclose{[}$ such that $V(w)=c$,
a contradiction to the fact that $V^{-1}(c)\subset B(p,R)$.
\end{proof}

By the above lemma, we can provide a new proof,
in the frame of the Brouwer fixed point theorem,
of a result previously obtained in \cite[Theorem~2]{Za-1996}.

\begin{corollary}\label{cor-3.2}
Let $V \colon \Omega\to \mathbb{R}$ (with $N\geq 2$) be a $\mathcal{C}^{1}$-function
and let $p\in \mathbb{R}^{N}$ and $c\in \mathbb{R}$ be such that
$V(p) < c$ with $V^{-1}(c)$ bounded and nonempty. Let also
$M:=V^{-1}(\mathopen{]}-\infty,c\mathclose{]})$.
Assume \eqref{eq-corollary} and \eqref{eq-rad}.
Then, there exists a $T$-periodic solution of \eqref{eq-3.1} with values in $M$.
\end{corollary}

\begin{proof}
We apply Lemma~\ref{lem-3.4} together with Corollary~\ref{cor-3.1} and Lemma~\ref{lem-3.2}, or, alternatively,
Theorem~\ref{th-starshaped}, jointly with Lemma~\ref{lem-3.4}.
\end{proof}

The above corollary provides an example of application of Theorem~\ref{th-starshaped}
to strictly star-shaped sets with the boundary described by a single regular
bounding function $V$. As a second application, we propose an application to domains
with possible non-smooth boundary, using a concept of outer normals due to
Bony \cite{Bo-1969}. Let $G$ be an open and bounded set with
$M:=\overline{G}\subset \Omega$ and let $u\in \partial M$. A vector $\nu\neq0$
is called an \textit{outer normal} to $M$ in $u$, according to Bony, if
$u$ is a point of $M$ at minimal distance from $u+\lambda\nu$ for some $\lambda >0$.
Playing on the coefficient $\lambda > 0$, one can equivalently express this fact,
by assuming that $B(u+ \nu,\|\nu\|) \cap M=\emptyset$, or
$B[u+ \nu,\|\nu\|] \cap M=\{u\}$ (actually, different but equivalent definitions
have been considered by some authors \cite{Bo-1969, Re-1972, ReWa-1975}).
We also denote by $N_{\mathrm{B}}(u)$ the set of all (Bony) outer normals to $M$ in $u$.
Then the following result holds \cite[Theorem~2.1]{Bo-1969} (see also
\cite[Theorem~1]{Re-1972} and \cite[Theorem~1]{ReWa-1975}).

\begin{theorem}\label{th-Bony}
Let $M:=\overline{G}\subset \Omega$, where $G$ is a nonempty open bounded set. Let $f=f(t,x)$ be continuous and locally Lipschitz
continuous in the $x$-variable.
Suppose that
\begin{equation}\label{eq-Bony}
\langle f(t,u),\nu \rangle \leq 0, \quad \text{for all $t\in \mathopen{[}0,T\mathclose{]}$, $u\in \partial M$, $\nu \in N_{\mathrm{B}}(u)$.}
\end{equation}
Then, $M$ is positively invariant with respect to the solutions of \eqref{eq-3.1}.
\end{theorem}

Clearly, under the assumptions of Theorem~\ref{th-Bony},
if $M$ has the (FPP) then it contains a $T$-periodic solution of
\eqref{eq-3.1}. In \cite{Re-1972,ReWa-1975} the Lipschitz condition was improved
to a suitable one-sided uniqueness hypothesis.
As in Theorem~\ref{th-cm}, the above result requires the inequality
\eqref{eq-Bony} to be satisfied \textit{for all} the outer normals at the boundary points.
Similarly as in Theorem~\ref{th-fz}, we propose now an existence result of
periodic solutions where, for the boundary condition, we assume that
the inequality is satisfied \textit{only for some} outer normals. We stress the fact that
only continuity of $f$ will be required. We restrict our application to star-shaped sets,
in order to enter in the setting of Theorem~\ref{th-starshaped};
however, our result, in principle, could be applied to more general domains,
provided we find a suitable auxiliary vector field $g$ as in Theorem~\ref{th-main}.

\begin{theorem}\label{th-starshaped1}
Let $G$ be an open bounded set with $\overline{G}(\subset \Omega)$
strictly star-shaped with respect to all the points in a ball
$B[p,\delta]\subset G$. Assume that for each $u\in \partial G$
there exists $\nu=\nu_{u}\in N_{\mathrm{B}}(u)$ such that
\begin{equation}\label{eq-Bony1}
\langle f(t,u),\nu_{u} \rangle \leq 0, \quad \text{for all $t\in \mathopen{[}0,T\mathclose{]}$.}
\end{equation}
Then, there exists a $T$-periodic solution of \eqref{eq-3.1} with values in $\overline{G}$.
\end{theorem}

\begin{proof}
By the hypothesis, for each $u\in \partial G$, we have an outer normal $\nu_{u}$ such that
\eqref{eq-Bony1} is satisfied. For such $\nu_{u}$ we have that
$B[u+ \nu_{u},\|\nu_{u}\|] \cap \overline{G}=\{u\}$ (according to one of the equivalent definitions of
Bony outer normal). Then, we define the function
\begin{equation*}
V_{u}(x):= \dfrac{1}{2} \bigl{(} \|\nu_{u}\|^{2} - \|x- (u + \nu_{u})\|^{2} \bigr{)}.
\end{equation*}
It is immediate to check that $V_{u}(u)=0$, $V_{u}(x)< 0$ for all $x\in G$, and
$\nabla V_{u}(x) = u + \nu_{u} - x$, so that $\nabla V_{u}(u) = \nu_{u}$. Hence \eqref{eq-3.4}
follows from \eqref{eq-Bony1} and we conclude by applying Theorem~\ref{th-starshaped}.
\end{proof}

\begin{remark}\label{rem-3.5}
Clearly Theorem~\ref{th-starshaped1} extends Theorem~\ref{th-cm1}. We also note that, in the context of
open and bounded star-shaped sets, the difference between Theorem~\ref{th-Bony} and Theorem~\ref{th-starshaped1}
is stronger than that between Theorem~\ref{th-cm} and Theorem~\ref{th-cm1}. Indeed, in the case of
convex bodies, all the points at the boundary possess outer normals. This is no more true for
star-shaped bodies. For instance, if we consider the sets
\begin{equation}\label{def-G1}
G_{1}:= \bigl{\{}(x_{1},x_{2})\in \mathbb{R}^{2} \colon (x_{1}^{2} + x_{2}^{2})<1, \, |x_{2}|>x_{1}\bigr{\}}
\end{equation}
and
$\overline{G_{1}}=\{(x_{1},x_{2})\in \mathbb{R}^{2} \colon (x_{1}^{2} + x_{2}^{2})\leq 1, \, (|x_{2}|\geq x_{1})\}$
(see Figure~\ref{fig-03}),
we have that $\overline{G_{1}}$ is strictly star-shaped with respect to all the points in a ball
$B[p,\delta]\subset G_{1}$ for $p=(-1/2,0)$ and $\delta >0$ sufficiently small. Observe that
$\overline{G_{1}}$ is homeomorphic to a closed disc by Lemma~\ref{lem-3.2}. Hence, if we assume $f$ locally
Lipschitz continuous in $x=(x_{1},x_{2})$ and \eqref{eq-Bony} satisfied \textit{for all} the outer normals
at $\partial G_{1}$ (whenever they exist), then, by Theorem~\ref{th-Bony}, we have the existence of a $T$-periodic solution
with values in $\overline{G_{1}}$. In this case, we have to check the sub-tangentiality condition
\eqref{eq-Bony} for all the outer normals at the corner points $(1/\sqrt{2},\pm 1/\sqrt{2})$
but no condition has to be checked at $(0,0)$, a point that does not possess any outer normal.
On the other hand, Theorem~\ref{th-starshaped1} cannot be applied because at $(0,0)$
there are no outer normals. Conversely, if we consider the sets
\begin{equation}\label{def-G2}
G_{2} := \bigl{\{}(x_{1},x_{2})\in \mathbb{R}^{2} \colon (x_{1}^{2} + x_{2}^{2})<1, \, (x_{1}-\sqrt{2})^{2} + x_{2}^{2}>1\bigr{\}}
\end{equation}
and
$\overline{G_{2}}=\{(x_{1},x_{2})\in \mathbb{R}^{2} \colon (x_{1}^{2} + x_{2}^{2})\leq 1, \, (x_{1}-\sqrt{2})^{2} + x_{2}^{2}\geq 1)\}$
(see Figure~\ref{fig-03}),
we have again that $\overline{G_{2}}$ is strictly star-shaped with respect to all the points in a ball
$B[p,\delta]\subset G_{2}$ for $p=(-1/2,0)$ and $\delta >0$ sufficiently small.
Now, to apply Theorem~\ref{th-starshaped1} we just need $f$ to be continuous and check the sub-tangentiality condition
at the boundary, by taking \textit{only one} outer normal at the corner points $(1/\sqrt{2},\pm 1/\sqrt{2})$.
\hfill$\lhd$
\end{remark}

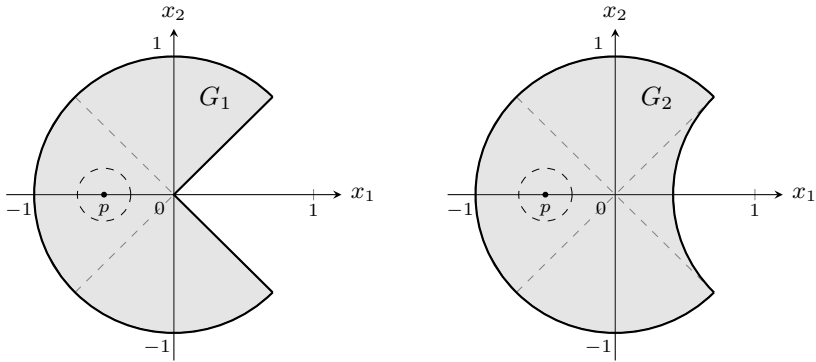
\begin{figure}[htb]
\centering
\begin{tikzpicture}
\begin{axis}[
  tick label style={font=\scriptsize},
  axis y line=middle,
  axis x line=middle,
  xtick={-1,1},
  ytick={-1,1},
  xticklabels={ , },
  yticklabels={ , },
  xlabel={\small $x_{1}$},
  ylabel={\small $x_{2}$},
every axis x label/.style={
    at={(ticklabel* cs:1.0)},
    anchor=west,
},
every axis y label/.style={
    at={(ticklabel* cs:1.0)},
    anchor=south,
},
  width=6cm,
  height=6cm,
  xmin=-1.2,
  xmax=1.2,
  ymin=-1.2,
  ymax=1.2]
\addplot  [color=gray, line width=0.2pt, dashed] coordinates {(-0.707107,-0.707107) (0.707107,0.707107)};
\addplot  [color=gray, line width=0.2pt, dashed] coordinates {(-0.707107,0.707107) (0.707107,-0.707107)};
\addplot [color=black, fill=gray, fill opacity=0.2, line width=0pt] coordinates {(0.707107, 0.707107) (0,0) (0.707107,-0.707107) (0.7, -0.714143) (0.6, -0.8) (0.5, -0.866025) (0.4, -0.916515) (0.3, -0.953939) (0.2, -0.979796) (0.1, -0.994987) (0., -1.) (-0.1, -0.994987) (-0.2, -0.979796) (-0.3, -0.953939) (-0.4, -0.916515) (-0.5, -0.866025) (-0.6, -0.8) (-0.7, -0.714143) (-0.8, -0.6) (-0.9, -0.43589) (-0.905, -0.425412) (-0.91, -0.414608) (-0.915, -0.403454) (-0.92, -0.391918) (-0.925, -0.379967) (-0.93, -0.36756) (-0.935, -0.354648) (-0.94, -0.341174) (-0.945, -0.32707) (-0.95, -0.31225) (-0.955, -0.296606) (-0.96, -0.28) (-0.965, -0.26225) (-0.97, -0.243105) (-0.975, -0.222205) (-0.98, -0.198997) (-0.985, -0.172554) (-0.99, -0.141067) (-0.995, -0.0998749) (-1., 0.) (-0.995, 0.0998749) (-0.99, 0.141067) (-0.985, 0.172554) (-0.98, 0.198997) (-0.975, 0.222205) (-0.97, 0.243105) (-0.965, 0.26225) (-0.96, 0.28) (-0.955, 0.296606) (-0.95, 0.31225) (-0.945, 0.32707) (-0.94, 0.341174) (-0.935, 0.354648) (-0.93, 0.36756) (-0.925, 0.379967) (-0.92, 0.391918) (-0.915, 0.403454) (-0.91, 0.414608) (-0.905, 0.425412) (-0.9, 0.43589) (-0.8, 0.6) (-0.7, 0.714143) (-0.6, 0.8) (-0.5, 0.866025) (-0.4, 0.916515) (-0.3, 0.953939) (-0.2, 0.979796) (-0.1, 0.994987) (0., 1.) (0.1, 0.994987) (0.2, 0.979796) (0.3, 0.953939) (0.4, 0.916515) (0.5, 0.866025) (0.6, 0.8) (0.7, 0.714143) (0.707107,0.707107)};
\fill (axis cs: -0.5,0) circle (1.1pt);
\draw [line width=0.2pt, dashed] (axis cs: -0.5,0) circle (10pt);
\draw [line width=0.8pt] (axis cs:0.707107,0.707107) arc [radius=52.4pt, start angle=45, end angle=315];
\addplot  [color=black, line width=0.8pt] coordinates {(0.707107,-0.707107) (0,0) (0.707107,0.707107)};
\node at (axis cs: 0.3,0.7) {$G_{1}$};
\node at (axis cs: -0.12,-1.1) {\scriptsize{$-1$}};
\node at (axis cs: -0.12,1.1) {\scriptsize{$1$}};
\node at (axis cs: -1.11,-0.11) {\scriptsize{$-1$}};
\node at (axis cs: 1,-0.11) {\scriptsize{$1$}};
\node at (axis cs: -0.5,-0.11) {\scriptsize{$p$}};
\node at (axis cs: -0.1,-0.1) {\scriptsize{$0$}};
\end{axis}
\end{tikzpicture}
\quad\quad
\begin{tikzpicture}
\begin{axis}[
  tick label style={font=\scriptsize},
  axis y line=middle,
  axis x line=middle,
  xtick={-1,1},
  ytick={-1,1},
  xticklabels={ , },
  yticklabels={ , },
  xlabel={\small $x_{1}$},
  ylabel={\small $x_{2}$},
every axis x label/.style={
    at={(ticklabel* cs:1.0)},
    anchor=west,
},
every axis y label/.style={
    at={(ticklabel* cs:1.0)},
    anchor=south,
},
  width=6cm,
  height=6cm,
  xmin=-1.2,
  xmax=1.2,
  ymin=-1.2,
  ymax=1.2]
\addplot  [color=gray, line width=0.2pt, dashed] coordinates {(-0.707107,-0.707107) (0.707107,0.707107)};
\addplot  [color=gray, line width=0.2pt, dashed] coordinates {(-0.707107,0.707107) (0.707107,-0.707107)};
\addplot [color=black, fill=gray, fill opacity=0.2, line width=0pt] coordinates {(0.707107,-0.707107) (0.7, -0.714143) (0.6, -0.8) (0.5, -0.866025) (0.4, -0.916515) (0.3, -0.953939) (0.2, -0.979796) (0.1, -0.994987) (0., -1.) (-0.1, -0.994987) (-0.2, -0.979796) (-0.3, -0.953939) (-0.4, -0.916515) (-0.5, -0.866025) (-0.6, -0.8) (-0.7, -0.714143) (-0.8, -0.6) (-0.9, -0.43589) (-0.905, -0.425412) (-0.91, -0.414608) (-0.915, -0.403454) (-0.92, -0.391918) (-0.925, -0.379967) (-0.93, -0.36756) (-0.935, -0.354648) (-0.94, -0.341174) (-0.945, -0.32707) (-0.95, -0.31225) (-0.955, -0.296606) (-0.96, -0.28) (-0.965, -0.26225) (-0.97, -0.243105) (-0.975, -0.222205) (-0.98, -0.198997) (-0.985, -0.172554) (-0.99, -0.141067) (-0.995, -0.0998749) (-1., 0.) (-0.995, 0.0998749) (-0.99, 0.141067) (-0.985, 0.172554) (-0.98, 0.198997) (-0.975, 0.222205) (-0.97, 0.243105) (-0.965, 0.26225) (-0.96, 0.28) (-0.955, 0.296606) (-0.95, 0.31225) (-0.945, 0.32707) (-0.94, 0.341174) (-0.935, 0.354648) (-0.93, 0.36756) (-0.925, 0.379967) (-0.92, 0.391918) (-0.915, 0.403454) (-0.91, 0.414608) (-0.905, 0.425412) (-0.9, 0.43589) (-0.8, 0.6) (-0.7, 0.714143) (-0.6, 0.8) (-0.5, 0.866025) (-0.4, 0.916515) (-0.3, 0.953939) (-0.2, 0.979796) (-0.1, 0.994987) (0., 1.) (0.1, 0.994987) (0.2, 0.979796) (0.3, 0.953939) (0.4, 0.916515) (0.5, 0.866025) (0.6, 0.8) (0.7, 0.714143) (0.707107,0.707107) (0.7, 0.699928) (0.65, 0.644963) (0.6, 0.580565) (0.590711, 0.567312) (0.581421, 0.553586) (0.572132, 0.53935) (0.562843, 0.524564) (0.553553, 0.50918) (0.544264, 0.493141) (0.534975, 0.476381) (0.525685, 0.458822) (0.516396, 0.440368) (0.507107, 0.420901) (0.497817, 0.400273) (0.488528, 0.378294) (0.479239, 0.354714) (0.469949, 0.329189) (0.46066, 0.301224) (0.451371, 0.270063) (0.442082, 0.234434) (0.432792, 0.191865) (0.423503, 0.135987) (0.414214, 0.) (0.423503, -0.135987) (0.432792, -0.191865) (0.442082, -0.234434) (0.451371, -0.270063) (0.46066, -0.301224) (0.469949, -0.329189) (0.479239, -0.354714) (0.488528, -0.378294) (0.497817, -0.400273) (0.507107, -0.420901) (0.516396, -0.440368) (0.525685, -0.458822) (0.534975, -0.476381) (0.544264, -0.493141) (0.553553, -0.50918) (0.562843, -0.524564) (0.572132, -0.53935) (0.581421, -0.553586) (0.590711, -0.567312) (0.6, -0.580565) (0.65, -0.644963) (0.7, -0.699928) (0.707107,-0.707107)};
\fill (axis cs: -0.5,0) circle (1.1pt);
\draw [line width=0.2pt, dashed] (axis cs: -0.5,0) circle (10pt);
\draw [line width=0.8pt] (axis cs:0.707107,0.707107) arc [radius=52.4pt, start angle=45, end angle=315];
\draw [line width=0.8pt] (axis cs:0.707107,0.707107) arc [radius=52.4pt, start angle=135, end angle=225];
\node at (axis cs: 0.3,0.7) {$G_{2}$};
\node at (axis cs: -0.12,-1.1) {\scriptsize{$-1$}};
\node at (axis cs: -0.12,1.1) {\scriptsize{$1$}};
\node at (axis cs: -1.11,-0.11) {\scriptsize{$-1$}};
\node at (axis cs: 1,-0.11) {\scriptsize{$1$}};
\node at (axis cs: -0.5,-0.11) {\scriptsize{$p$}};
\node at (axis cs: -0.1,-0.1) {\scriptsize{$0$}};
\end{axis}
\end{tikzpicture}\captionof{figure}{Representation of the set $G_{1}$
defined in \eqref{def-G1} (on the left) and of the set $G_{2}$
defined in \eqref{def-G2} (on the right).}
\label{fig-03}
\end{figure}

\section*{Acknowledgments}
The authors thank Prof.~Jean Mawhin for the interesting
and helpful discussions on the subject of the present paper.

\bibliographystyle{elsart-num-sort}
\bibliography{FeZa-biblio-2022}

\end{document}